\documentclass[12pt]{amsart}

\usepackage{amsthm}
\usepackage{graphicx}
\usepackage{amssymb}

\usepackage[font=small,labelfont=bf]{caption}
\usepackage{multirow}

\newcommand{\C}{\mathcal{C}}

\newcommand{\N}{\mathbb{N}}

\renewcommand{\epsilon}{\varepsilon}
\newcommand{\mys}{\ensuremath{|}}
\newcommand{\myd}{\ensuremath{||}}

\begin{document}

\newtheorem{theorem}{Theorem}[section]
\newtheorem{prop}[theorem]{Proposition}
\newtheorem{corollary}[theorem]{Corollary}


\title[Quadrangulated Immersions of Cubic Graphs]{Which Cubic Graphs have Quadrangulated Spherical Immersions?}
\author[L. Abrams]{Lowell Abrams}
\curraddr[L. Abrams]{University Writing Program and Department of Mathematics\\The George Washington University\\Washington, D.C. USA}
\email{labrams@gwu.edu}

\author[Y. Berman]{Yosef Berman}
\curraddr[Y. Berman]{New York, NY USA}
\email{berman.yosef@gmail.com}

\author[V. Faber]{Vance Faber}
\curraddr[V. Faber]{Center for Computing Sciences\\ Bowie, MD USA}
\email{vance.faber@gmail.com}

\author[M. Murphy]{Michael Murphy}
\curraddr[M. Murphy]{Venice, CA USA}
\email{piggymurph@gmail.com}

\date{\today}

\begin{abstract}We consider spherical quadrangulations --  spherical embeddings of multigraphs, possibly with loops, so that every face has boundary walk of length 4 -- in which all vertices have degree 3 or 4. Interpreting each degree 4 vertex as a crossing, these embeddings can also be thought of as transversal immersions of cubic graphs which we refer to as the {\it extracted graphs}. We also consider quadrangulations of the disk in which interior vertices have degree 3 or 4 and boundary vertices have degree 2 or 3. First, we classify all such quadrangulations of the disk. Then, we  provide four methods for constructing spherical quadrangulations, two of which use quadrangulations of the disk as input. Two of these methods provide one-parameter families of quadrangulations, for which we prove that the sequence of isomorphism types of extracted graphs is periodic. We close with a description of computer computations which yielded spherical quadrangulations for all but three cubic multigraphs on eight vertices. 
\end{abstract}


\maketitle

\section{Introduction}

This paper concerns spherical quadrangulations --  spherical embeddings of multigraphs, possibly with loops, so that every face has boundary walk of length 4 -- in which all vertices have degree 3 or 4.  (In this paper, the term ``graph'' allows for multiple edges and loops.)  In fact, a simple Euler characteristic argument shows that there must be exactly 8 cubic vertices. We view such quadrangulations as immersions of cubic graphs with transverse crossings, and study their properties and various construction methods, while considering the question of which cubic graphs can be so realized. As a first set of examples, Figure \ref{fig:connected} shows quadrangulated spherical immersions of the five simple connected graphs on 8 vertices. 

\begin{figure}
\begin{center}
\includegraphics[width=.97\textwidth]{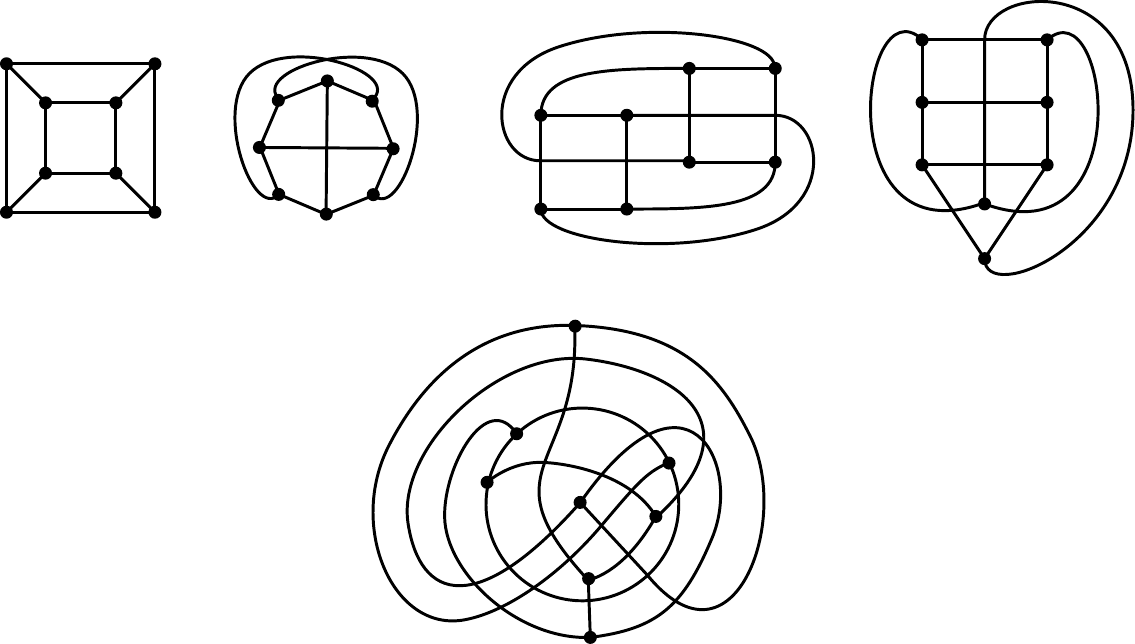}
\end{center}
\caption{Quadrangulated immersions of the five simple connected cubic graphs on eight vertices.}
\label{fig:connected}
\end{figure}

The work here is an outgrowth and development of preliminary ideas of Abrams and Slilaty \cite{AbramsSlilaty:Grids}. Methods for constructing spherical quadrangulations with all vertices of degree 3 or 4 are also developed by Hasheminezhad and McKay  \cite{HasheminezhadMcKay}, but they do not interpret such embeddings as immersions, and the interaction of their methods with the notion of immersion is not immediately clear.

Quadrangulations are of interest because they concentrate all ``curvature'' at vertices, thereby giving a fair bit of control over their structure. For instance, a very explicit characterization of quadrangulations of the [flat] torus and [flat] Klein bottle in which every vertex has degree 4 (along with the additional property that the four faces around each vertex along with their boundaries form a $2\times
2$ square grid) was initially given by Thomassen \cite{Thomassen91}; a slightly different formulation is given by
M\'arquez, de Mier, Noy, and Revuelta \cite{MarquezDeMierNoy03}.  In any graph embedding, if ``most" of the vertices are of degree 4, then $G$ has ``large" areas that are annular or appear as the standard, geometrically-flat, infinite $\{4,4\}$-planar lattice. In contrast to this, vertices that are not of degree 4 create the curvature necessary for a quadrangulation to be in a surface other than the torus or Klein bottle.

Although our objects of interest are immersions of graphs in the sphere, it is helpful to also think in terms of immersions of one graph in another. Let $P(G)$ denote the set of paths in $G$. Formally, an {\it immersion} of a graph $H$ in a graph $G$ is a function $\beta\colon V(H) \sqcup E(H) \to V(G) \sqcup P(G)$ for which 
\begin{itemize}
\item $\beta(V(H)) \subseteq V(G)$;
\item $\beta(E(H)) \subseteq P(G)$;
\item if $e=uv\in E(H)$ then $\beta(e)$ is a $\beta(u),\beta(v)$-path in $G$;
\item if $e,e'\in E(H)$ are distinct edges then the paths $\beta(e)$ and $\beta(e')$ are edge-disjoint.
\end{itemize}
Our case actually demands more: an immersion of $H$ in $G$ is a {\it strong immersion} if for any edge $uv\in E(H)$, the path $\beta(uv)$ intersects $\beta(V(H))$ only at the ends of $\beta(uv)$. Thus, our objects of interest are strong graph immersions $\beta$ of a cubic graph $H$ on 8 vertices into a spherical quadrangulation $G$ for which $\beta\mid_{V(H)}$ is a bijection.

An additional restriction we place on our immersions is that crossings must be transversal. This can be described combinatorially using language of Malkevitch, who defines a notion of {\it coded path} which, when approaching a given vertex, indicates the number of edges to rotate by before continuing on \cite{Malkevitch70}. We are considering paths which begin and end at cubic vertices, have all interior vertices of degree 4, and have the code ``$(2)$'', {i.e.}, the continuation through any degree 4 vertex $v$ is never to an edge which is adjacent to the entry-edge in the rotational order at $v$.

Section \ref{sec:embeddings and immersions} presents the basic notions of embeddings and immersions we need, and in particular highlights the notion of {\it transversal extension of a path}. This is followed in Section \ref{sec:cubicquads} by the definition of our main object of study, which  as we have outlined above is a cellular spherical embedding such that all faces are quadrangles and all vertices have degree 3 or 4; for brevity we will refer to these simply as {\it quadrangulated immersions}.  Our first major theorem is Theorem \ref{thm:quadrangulated_disks}, which classifies all quadrangulated immersions in a disk; these are the building blocks for some of the constructions we provide. 

Section \ref{sec:constructions} presents four techniques for constructing quadrangulated immersions, along with appropriate verifications of correctness:
\begin{itemize}
\item[5.1.] The two disks construction, which produces quadrangulated immersions by attaching two quadrangulated disks along their boundary;

\item[5.2.] The radial construction, which makes use of the fact that the radial graph of a quadrangulated immersions is itself a quadrangulated immersion;

\item[5.3.] The spiral construction, which begins with a quadrangulated disk and builds the remainder of the quadrangulated immersions using a specific ``spiraling'' process.  Proposition \ref{pr:spiralperiodicity} identifies a particular periodicity property that this construction enjoys;

\item[5.4.] The cable construction, which is similar in flavor to the spiral construction, but works by modifying an existing quadrangulated immersions. Again, we have a periodicity result, given in Theorem \ref{thm:cable}.
\end{itemize}

In Section \ref{sec:computations} we discuss computational efforts to produce quadrangulated spherical immersions of all cubic multigraphs on 8 vertices. We found  immersions of all 69 non-connected multigraphs on 8 vertices, and of the 71 connected multigraphs we found  immersions of all but the three shown in Figure \ref{fig:outlaws}. We comjecture that, in fact, these three graphs have no quadrangulated immersion.

\begin{figure}
\begin{center}
\includegraphics[width=.8\textwidth]{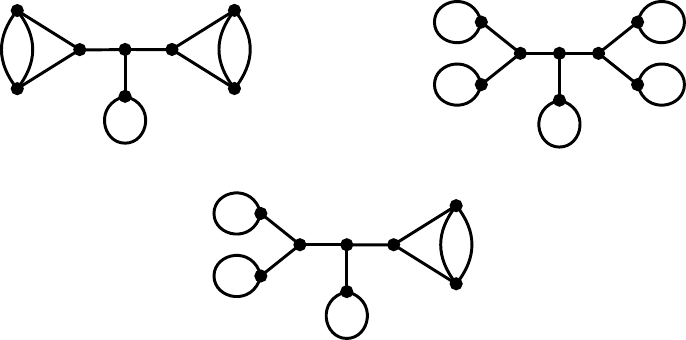}
\end{center}
\caption{The three cubic graphs on 8 vertices for which we do not have a quadrangulated immersion. In the notation from Section \ref{subsec:cubic_multigraphs}, these are 3a\mys 2b\mys 3a, 3a\mys 2b\mys 3b, 3b\mys 2b\mys 3b .}
\label{fig:outlaws}
\end{figure}

\section{Embeddings, Transversals, and Immersions} \label{sec:embeddings and immersions}

We refer to \cite{GrossTucker} for the basics of graph embeddings not covered here. A graph $G$ embedded in a topological surface $S$ is said to be {\it cellularly embedded} if the complement of $G$ in $S$ is a disjoint union of open topological disks; these are the {\it faces} of the embedding. A cellular graph embedding in an orientable surface corresponds to a {\it rotation scheme}, {\it i.e.}, an assignment to each vertex of a cyclic ordering of the edges incident to that vertex. 

If graph $G$ is embedded in a surface $S$ and $P$ is a path in $G$ with vertices $v_0, v_1, \ldots, v_n$, in that order, then we say $P$ is a {\it transverse path of $G$ in $S$} if, for $i=1,2, \ldots, n-1$, we have $\deg(v_i) = 4$  and the edges of $P$ incident to $v_i$ are non-consecutive in the rotation scheme. A transverse path $P$ is {\it complete} if neither of its endpoints has degree 4 in $G$. A {\it transverse closed walk of $G$ in $S$} (or just {\it transversal of $G$ in $S$ }) is a closed walk $W$ of $G$ such that for each vertex $v$ in $W$ we have $\deg(v)=4$ and if $w,w'$ are consecutive edges of $W$ incident to $v$ then $w,w'$ are non-consecutive in the rotation scheme.

Given a cellular embedding of a graph $G$ for which all vertices of even degree have degree 4, we may construct another graph $\epsilon(G)$ we call the {\it extraction from $G$} as follows. The vertices of $\epsilon(G)$ are the odd-degree vertices of $G$, and given two odd-degree vertices $u,v$ of $G$ there is an edge $uv$ in $\epsilon(G)$ exactly when there is a transverse path in the embedding of $G$ in $S$ from $u$ to $v$. Note that embedding of $G$ in $S$ corresponds to an immersion of $\epsilon(G)$ in $S$, and the vertices of degree 4 in $G$ correspond to transverse crossings in the immersion.

Consider a cellular embedding of a graph $G$. Suppose that $f$ is a 5-sided face in this embedding having boundary vertices $v_0, v_1, \ldots, v_4$ (indexed modulo 5), in that cyclic order, and that $P$ is a path in $G$ whose last vertex is $v_i$ for some $i$, but whose last edge is not on $\partial f$. A { \it single-edge transverse extension of $P$} is the path obtained by subdividing the edge $v_{i+2}v_{i-2}$ with a new vertex $u$, then introducing a new edge  $v_iu$ to the embedding and adjoining it to $P$; see Figure \ref{fig:transextension}. Note that the face $f$ has now been subdivided into two quadrangles, and that if the face other than $f$ incident to $v_{i-2}v_{i+2}$ had originally been a quadrangle it is now a pentagon. A {\it transverse extension of $P$} is the path obtained by iteratively performing single-edge transverse extensions. Intuitively, one may think of a transverse extension as being obtained by extending $P$ ``straight across'' a sequence of quadrangles. 
\begin{figure}
\begin{center}
\includegraphics[width=.8\textwidth]{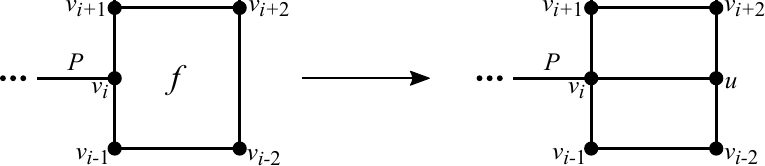}
\end{center}
\caption{A single-edge transverse extension.}
\label{fig:transextension}
\end{figure}

\section{Quadrangulated Immersions} \label{sec:cubicquads}

A {\it quadrangulated immersion} is a cellular embedding of a graph $G$ in the sphere $S^2$ such that each vertex of $G$ either has degree 3 or degree 4, and all faces of the embedding are quadrangles. Note that this can also be thought of as an immersion of the cubic graph $\epsilon(G)$. 

For a  quadrangulated immersion $G$, let $\nu, \mu$, and $\lambda$ denote the quantity of vertices, edges, and faces, respectively. Also, let $\nu_i$ denote the quantity of vertices  of degree $i$. 
\begin{prop}
In any quadrangulated immersion, $\nu_3 =  8$, and $\lambda= 6+\nu_4$. 
\end{prop}
\begin{proof}
We have $2\mu = 3\nu_3 + 4\nu_4$ and $2\mu = 4\lambda$. Taken together with the Euler formula $(\nu_3+\nu_4)-\mu+\lambda=2$, we get
\[
2 = (\nu_3+\nu_4) - \left(\frac{3}{2}\nu_3 + 2\nu_4\right) + \left(\frac{3}{4}\nu_3+\nu_4\right) = \frac{1}{4}\nu_3,
\]
which give $\nu_3 = 8$. We also have $4\lambda = 3\nu_3 + 4\nu_4 = 24 + 4\nu_4$, so $\lambda = 6+\nu_4$. 
\end{proof}

It is well known that any graph embedding in the sphere having all facial boundaries of even length, such as a quadrangulated immersion, is a bipartite graph. If $A,B \subset V(G)$ give a bipartition of $G$, then we write $\nu_{A,i} $ (or $\nu_{B,i}$) for the quantity of vertices in $A$ (or $B$) having degree $i$. 
\begin{prop} \label{pr:bipart_vcounts}
Let $G$ in $S^2$ be a quadrangulated immersion with bipartition $A\cup B$. Then
\[
3\nu_{A,3} + 2\left(\nu_{A,4}-\nu_{B,4}\right) = 12,
\]
and in particular, $\nu_{A,3} = 0 \mod 2$ and $\nu_{A,4} = \nu_{B,4}  \mod 3$. 
\end{prop}
\begin{proof}
Counting edges by vertex degrees, we obtain
\[
3\nu_{A,3} + 4\nu_{A,4} = 3\nu_{B,3}+4\nu_{B,4} = 3\left(8 - \nu_{A,3}\right) + 4\nu_{B,4},
\]
which simplifies to 
\[
3\nu_{A,3} + 2\left(\nu_{A,4}-\nu_{B,4}\right) = 12,
\]
from which the additional relations follow by reducing modulo 2 and 3, respectively. 
\end{proof}

\begin{corollary} \label{cor:6and2}
If a quadrangulated immersion has $\nu = 1 \mod 2$ then it has 6 vertices of degree 3 in one class of the bipartition and 2 in the other. 
\end{corollary}
\begin{proof}
Assume, without loss of generality, that $\nu_{A,3} \geq \nu_{B,3}$. Since $\nu_{A,3} = 0 \mod 2$, we see that $\nu_{A,3} \in \{4,6,8\}$, and from  Proposition \ref{pr:bipart_vcounts} we obtain  
\[
\nu = 8 + \nu_{A,4} + \nu_{B,4} = 8 + \nu_{A,4} + \frac{3}{2}\nu_{A,3} + \nu_{A,4} - 6 = 2\nu_{A,4} + 2 +  \frac{3}{2}\nu_{A,3}.
\]
Since this is even when $\nu_{A,3} \in \{4,8\}$ and odd when $\nu_{A,3}=6$, we are done.
\end{proof}

Proposition \ref{pr:one_crossing} follows the theme of a proposition of Blind and Blind  \cite{BlindBlind94}. 
\begin{prop}\label{pr:one_crossing}There is no quadrangular immersion of a simple cubic graph having exactly one crossing, and there is only one having two crossings.
\end{prop}
Another way to state Proposition \ref{pr:one_crossing} is that a quadrangulated immersion whose extracted graph is simple cannot have exactly 9 vertices, and there is exactly one such quadrangulated immersion having exactly 10 vertices. 
\begin{figure}
\begin{center}
\includegraphics{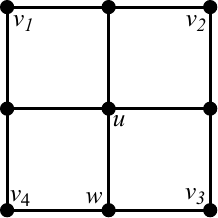}
\end{center}
\caption{The local configuration for a single degree-4 vertex.}
\label{fig:one_crossing}
\end{figure}
\begin{proof}
If a simple quadrangulated immersion $G$ has a single degree-4 vertex $u$, then by Corollary \ref{cor:6and2} $G$ has six vertices of degree 3 in one class of the bipartition and 2 in the other. Even with $u$ included in the second class, there would be 18 edges leaving the first class but only 10 leaving the second class, a contradiction since $G$ is simple.

\begin{figure}
\begin{center}
\includegraphics[height=2in]{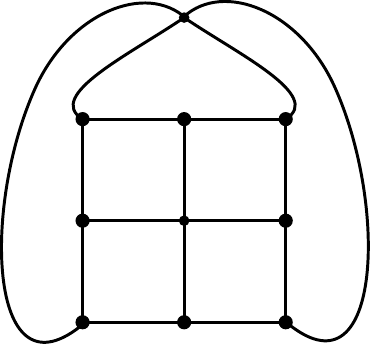}
\end{center}
\caption{The quadrangulated immersion with 10 vertices.}
\label{fig:two_crossings}
\end{figure}
Suppose now that $G$ has two degree-4 vertices, $u$ and $w$. As before, $u$ must have a neighborhood as shown in Figure \ref{fig:one_crossing}, with none of the boundary vertices identified. Thus, in this case,  Figure \ref{fig:one_crossing} shows all but one vertex of $G$ and all but four edges of $G$. If the unshown vertex has degree 3, then there is an edge connecting two of the boundary vertices. Up to dihedral symmetry, the only edge connecting two boundary vertices which creates no parallel edge and no odd cycle is $v_1w$. However, this edge cause the degree-2 vertices on the boundary to be in separate faces, implying that there is nowhere the remaining degree 3 vertex can be placed. It follows that the unshown vertex has degree 4, and the quadrangulated immersion is as shown in Figure \ref{fig:two_crossings}.
\end{proof}

\section{Quadrangulated Disks}

Given a disk $D$ write $\partial D$ to denote its boundary and $D^\circ$ to denote its interior. We say a connected graph $G$ is {\it properly embedded} in a disk $D$ if $\partial D$ is covered by a cycle of $G$. We say that $D$ is {\it quadrangulated by $G$} if $G$ is properly embedded in $D$ and every face of the embedding is a quadrangle. If $G$ is properly embedded in $D$, let $b_k$ denote the quantity of boundary vertices of degree $k$ and let $i_k$ denote the quantity of interior vertices of degree $k$.

\begin{prop}\label{pr:b+i=4}
If disk $D$ is quadrangulated by graph $G$, every vertex on $\partial D$ is of degree 2 or 3, and every vertex in $D^\circ$ is of degree 3 or 4, then $b_2+i_3 = 4$. 
\end{prop}
\begin{proof}
Letting $e$ denote the quantity of edges in $G$ and $f$ the quantity of [interior] faces, we have
\[
2e = 2b_2 + 3(b_3+i_3)+4i_4 \mbox{ and } 2e = 4f + (b_2+b_3).
\]
Using these relations together with the Euler characteristic formula $(b_2+b_3+i_3+i_4) - e + f = 1$, the result readily follows.
\end{proof}
We refer to a quadrangulated disk as a {\it digon, triangle} or {\it square} according as  $b_2=2,3$ or $4$, respectively, and we refer to the vertices of degree 2 on the boundary of the disk as its {\it corners}.

Given a quadrangulation of a disk $D$ by a graph $G$ we construct a {\it corresponding buffered quadrangulation} of a disk $D'$ having no vertices of degree 2 on $\partial D'$ as follows: Let the vertices of $G$ on $\partial D$ be $v_1, \ldots, v_n$, in that cyclic order. Copy the embedding of $G$ in $D$ into  $D'^\circ$, place an $n$-cycle with vertices $v'_1, \ldots, v'_n$ , in that cyclic order, on $\partial D'$, and for each $i$ add an edge $v_1v'_1$ so as to form a ring of quadrangles incident to $\partial D'$; see Figure \ref{fig:buffer}. Observe that the respective degrees of the vertices of $G$ on $\partial D$ have now each increased by 1. 
\begin{figure}
\begin{center}
\includegraphics[width=.7\textwidth]{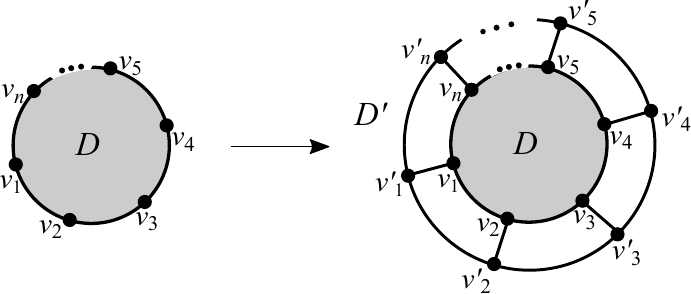}
\end{center}
\caption{Construction of the buffered quadrangulation.}
\label{fig:buffer}
\end{figure}

Given a vertex $v$ of degree 2 with incident edges $uv$ and $vw$, to {\it smooth the vertex $v$} means to add an edge $uw$ and delete $v$ along with $uv$ and $vw$. Suppose disk $D$ is quadrangulated by graph $G$.  If $T$ is a transversal $W$ or is a  transverse path $P$ having both ends on $\partial D$, then $G$ can be simplified by deleting all edges of $T$ and smoothing all resulting degree 2 vertices; because $\partial D$ is covered by a cycle of $G$, the resulting embedding is still cellular. We refer to the cellular embedding $\rho(G)$ obtained by deleting all closed  transverse walks in $G$ the {\it reduction of $G$}, and we say the quadrangulation is {\it irreducible} if $G = \rho(G)$. We refer to a quadrangulation of $D$ as a {\it quadrangulated immersion} if every vertex in $D^\circ$ has degree 3 or 4 and every vertex in $\partial D$ has degree 2 or 3.

\begin{theorem} \label{thm:quadrangulated_disks}
All irreducible quadrangulated immersions of a disk are isomorphic to one of the quadrangulations given in Figure \ref{fig:disks}, or to the buffered quadrangulation corresponding to one of those.
\end{theorem}
\begin{figure}
\begin{center}
\includegraphics[width=.9\textwidth]{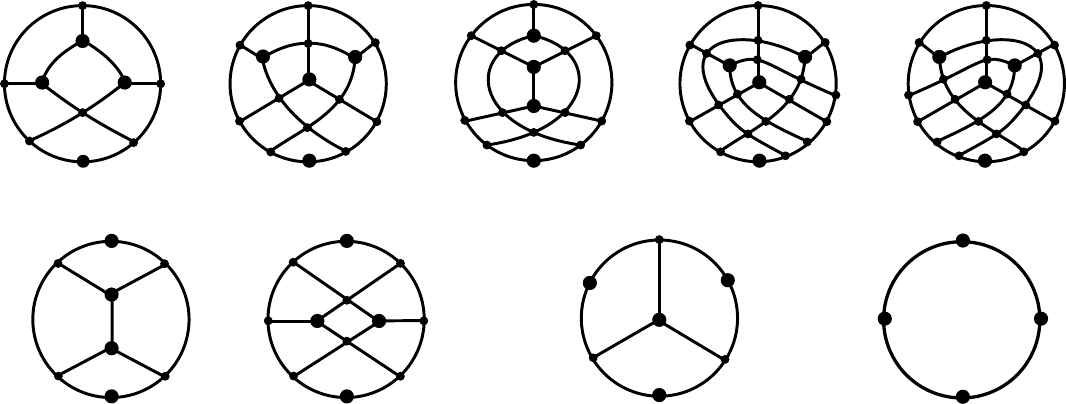}
\end{center}
\caption{The irreducible quadrangulated immersions of a disk with at least one degree-2 vertex on the boundary.}
\label{fig:disks}
\end{figure}
\begin{proof}
Suppose disk $D$ is irreducibly quadrangulated by graph $G$. Let $b_2$ and $i_3$ be defined as above; by Proposition \ref{pr:b+i=4} these satisfy $b_2 + i_3 = 4$. Note that for each of the $i_3$ degree-3 vertices $v$ in $D^\circ$ there are three transverse paths with one end on $v$, and because the quadrangulation is irreducible, all edges in $G$ belong to one of these paths. The strategy of this proof is to study the way such transverse paths can be arranged in $D$. 

We begin with an observation that will reduce our work later. Suppose that $v$ is a degree-3 vertex in $D^\circ$ and  that $P_1$ and $P_2$ are transverse paths with one end at $v$. If $P_1$ and $P_2$ share a single vertex $w\neq v$, then the portions of $P_1$ and $P_2$ extending between $v$ and $w$ bound a digon $D'$, and thus by Proposition \ref{pr:b+i=4} there must be exactly two degree-3 vertices in $D'^\circ$. We see that if $w$ has degree 3 in $G$ then we are in the case $i_3 = 4$, and otherwise we are in the case $i_3 =3$. Similarly, it is only possible to have a transverse path cross itself in the case $i_3 = 4$.

We now break up our analysis according to the possible values of $i_3$. 

\smallskip
\paragraph{\bf Case $\mathbf{i_3 = 1}$.} Let $v$ be the degree-3 vertex in $D^\circ$ and let $P_1,P_2,$ and $P_3$ be the transverse paths starting at $v$. As observed above, these paths must all extend from $v$ to $\partial D$ without intersecting each other, and thus in fact each consists of a single edge. Let $u_1,u_2,u_3$ be the ends of $P_1,P_2,P_3$, respectively, on $\partial D$; these three vertices and the three degree-2 vertices account for all vertices of $G$ on $\partial D$. If $u_i, u_j$ are adjacent on $\partial D$ for some $i,j$, then $P_i,P_j$ and the edge $u_iu_j$ bound a triangle $D'$ with corners $v,u_i$, and $u_j$, necessitating the presence of an additional degree-3 vertex $w$ in $D'^\circ$. Since $w$ would be in $D^\circ$, this contradicts the assumption that $i_3=1$. We conclude that $G$ in $D$ has the structure shown on the bottom row of Figure \ref{fig:disks}, to the right of center.

\smallskip
\paragraph{\bf Case $\mathbf{i_3 = 2}$.} Let $v_1, v_2$ be the degree-3 vertices in $D^\circ$, and let $u_1,u_2$ be the degree-2 vertices on $\partial D$.

Suppose first that \ transverse path from $v_1$ to $v_2$, so that all transverse paths from $v_1$ or from $v_2$ have their other end on $\partial D$. If the ends of all three transverse paths $P_1,P_2,P_3$ from $v_1$  lie on a single $u_1,u_2$-arc of $\partial D$, then $P_1,P_2,P_3$ together with portions of the boundary arc form two triangles, necessitating the presence of an additional two vertices of degree 3. Since there is only one degree 3 vertex other than $v_1$, this would be a contradiction. Thus two of $P_1,P_2,P_3$ end on one $u_1,u_2$-arc, and one ends on the other. We see that $P_1,P_2$ and $P_3$, together with appropriate portions of the boundary, form two squares and one triangle. This forces $v_2$ to be located in the interior of the triangle as in the case $i_3=1$, and for its transversely extended transverse paths to be arranged as shown on the bottom row of Figure \ref{fig:disks}, to the left of center.

Now suppose that there is a  transverse path $P$ from $v_1$ to $v_2$. By the work-saving observation above the path $P$ cannot cross any other transverse path. Let $P_1,P_2$ be the two transverse paths from $v_1$ other than $P$. If $P_1$ and $P_2$ end on the same  $u_1,u_2$-arc then they form a triangle $D'$, so $v_2$ must be located in $D'^\circ$. But this is the same as the case $i_3=1$, so we see that $P$ must cross either $P_1$ or $P_2$, a contradiction. It follows that $P_1$ and $P_2$ end on different  $u_1,u_2$-arcs, and by an analogous argument this is true of the two transverse paths from $v_2$. Finally, if a transverse path from $v_2$ crosses $P_1$ or $P_2$ then, together with $P$ or a portion of a  $u_1,u_2$-arc, we necessarily have a digon or a triangle, a contradiction. Thus, the configuration must be  as shown on the bottom row of Figure \ref{fig:disks}, to the far left.

\smallskip
\paragraph{\bf Case $\mathbf{i_3 = 3}$.} Let $u$ be the degree-2 vertex on $\partial D$ and let $v_1,v_2,v_3$ be the degree-3 vertices in $D^\circ$. In this case, it is possible that there is a vertex $v_i$ for which there are two transverse paths from $v_i$ that cross. In Case A we assume there is no such crossing, and in Case B we assume that there is.

\smallskip
\noindent{\bf Case A.\ } Suppose first that one of $v_1, v_2,$ and $v_3$ has no  transverse path to either of the other two; without loss of generality say this is $v_1$.  Then the  transverse paths from $v_1$, together with $\partial D$, form two triangles and, corresponding to the portion of $\partial D$ containing the degree-2 vertex $u$, a square. By Proposition \ref{pr:b+i=4}, one triangle, call it $T_2$, must contain $v_2$ and the other, call it $T_3$, contains $v_3$. Let $P$ denote the transverse path on the boundary of both $T_2$ and $T_3$, and let $P_2, P_3$ denote the other transverse paths from $v_1$ that are on the boundary of $T_2,T_3$, respectively. These are configured as in the following diagram.
\begin{center}
\includegraphics[height=1in]{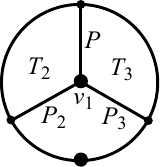}
\end{center}
The transverse paths from $v_2$ and from $v_3$ must extend to the boundary arcs  of $T_2$ and $T_3$, respectively, as in Case $i_3=1$. Those paths intersecting $P_2$ and $P_3$ have unique transverse extensions, and those extensions end at $\partial D$. Let $u_2$ and $u_3$ be the vertices at which  transverse paths $P'_2$ and $P'_3$ from $v_2$ and $v_3$, respectively, [first] intersect $P$. If $u_2\neq u_3$, then there unique transverse extensions of $P'_2$ and $P'_3$, which end at $\partial D$, yielding one of the two configurations on the top right of Figure \ref{fig:disks}, which violate the case A hypothesis for paths from $v_2$ or $v_3$, depending on the specific configuration. (In fact, the specific configuration depending on whether $u_2$ or $u_3$ was closer along $P$ to $v_1$.  If $u_2 = u_3$, then we obtain the configuration on the top of Figure \ref{fig:disks}, to the left of center. )

Suppose now that from each of $v_1,v_2,$ and $v_3$ there is a transverse path to at least one of the others. Suppose first, for the sake of contradiction, that there are three such paths $P_1,P_2,P_3$ connecting $v_1,v_2,$ and $v_3$ in a cyclic fashion (but not necessarily in that cyclic order). It cannot be that any of $P_1,P_2,P_3$ cross any of the others, as that would contradict our Case A hypothesis, and it cannot be that any of $P_1,P_2,P_3$ crosses itself, as we have $i_3 < 4$. This means that $P_1,P_2,P_3$ bound a triangle, again a contradiction since $i_3 =3$. 

It follows that $v_1,v_2,v_3$ are connected by transverse paths in a linear fashion, say by a transverse path $P_1$ from $v_1$ to $v_2$ and a transverse path $P_2$ from $v_2$ to $v_3$. Because of our Case A hypothesis and the impossibility of having any triangles in the $D^\circ$ which do not contain a degree-3 vertex, we must have the partial configuration shown below on the left:
\begin{center}
\includegraphics[height=1in]{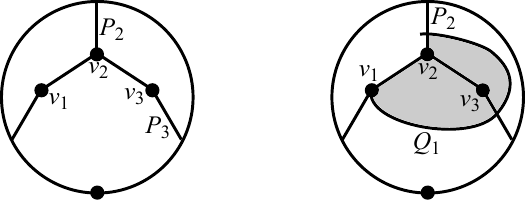}
\end{center}
Let $Q_1Q_3$ denote the third transverse paths from $v_1,v_3$, respectively, not shown on the left above. Again invoking our Case A hypothesis and the issue of triangles, if $Q_1$ crosses any of the other transverse paths shown, it must first cross $P_3$ and then, because a triangle may not be formed, it must cross $P_2$ as well. This yields the configuration shown above on the right. Note that the two shaded regions are squares, and therefore $Q_3$ necessarily forms a triangle, a contradiction. It follows that $Q_1$ does not cross any of the other transverse paths shown, and an analogous argument shows the same holds of $Q_3$. It now readily follows that we must have the configuration shown on the top left of Figure \ref{fig:disks}.

\smallskip
\noindent{\bf Case B.\ } We now suppose that there are two crossing transverse paths from a single degree-3 vertex, say $v_1$. This creates a digon, which necessarily contains $v_2$ and $v_3$ in its interior. To avoid creating any triangles adjacent to $\partial D$, we must have the following partial configuration:
\begin{center}
\includegraphics[height=1in]{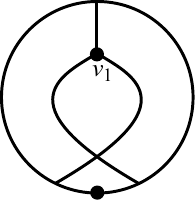}
\end{center}
Applying Case $i_3=2$, if $v_2$ and $v_3$ have a transverse path between them, then we obtain the configuration in the center of the top row of Figure \ref{fig:disks}, and if not then we obtain one of the two configuration on the right side of the top row. (See Case A for some comments about these.)

\smallskip
\paragraph{\bf Case $\mathbf{i_3 = 4}$.} Let the vertices on $\partial D$ be $u_1, u_2, \ldots, u_n$, in that cyclic order. Since $b_2 = 4 - i_3 = 0$, for each $i$ there is a vertex $w_i$ and an edge $u_iw_i$ not on $\partial D$. Moreover, because $D$ is quadrangulated, for each $i$ the edges $w_iu_i, u_iu_{i+1},u_{i+1}w_{i+1}$ are three sides of a quadrangle, and therefore $w_i \neq w_{i+1}$ and there is an edge $w_iw_{i+1}$.

Now, for the sake of contradiction, suppose that $w_i=w_j$ for some $i\neq j$. We then find four edges $w_iu_i, w_iw_{i-1}, w_iw_{i+1}$, and $w_iu_j$. Since $w_i$ has degree at most 4, there must be exactly four quadrangles at $w_i$; taking into consideration the orientation of the cycle on $\partial D$, in cyclic order around $w_i$ these are $u_{i-1}u_{i}w_iw_{i-1}$,   $u_{i}u_{i+1}w_{i+1}w_{i}$,  $u_{j-1}u_{j}w_{i}w_{i+1}$, and  $u_{j}u_{j+1}w_{i-1}w_{i}$. This implies, though, that the path with vertices $u_i, w_i=w_j, u_j$ is a transversal path from $\partial D$ to $\partial D$, contradicting the assumption that the configuration in question is irreducible. It follows that the vertices $w_1, w_2, \ldots, w_n$ are distinct and so, in that order, form a cycle of $G$. Since we have an irreducible embedding such a cycle is not a transversal, and therefore  bounds a disk $D'$ with at least one degree 2 vertex (degree 3 in $G$) on $\partial D'$, hence at most three degree-3 vertices in $D'^\circ$, confirming that $D$ is a buffering of one of the configurations identified in the previous cases.
\end{proof}

\section{Constructions} \label{sec:constructions}

\subsection{The two disks construction}

A simple way to construct quadrangulated immersions is to attach two quadrangulated disks along their boundaries. Given quadrangulated disks $D$ and $D'$ with a bijection $\varphi\colon V(\partial D)\to V(\partial D')$ which preserves the cyclic ordering of vertices (although not necessarily orientation) and such that either $\deg(v)\neq 2$ or $\deg\phi(v)\neq2$, we can attach $D$ to $D'$ by identifying $v$ with $\varphi(v)$ for each $v\in V(\partial D)$ and deleting the edges of $\partial D$; write $D \cup_\varphi D'$ for the resulting quadrangulated immersion. Note that neither of the quadrangulations of $D$ and $D'$, respectively, needs to be irreducible. Moreover, it is possible that $D \cup_\varphi D'$ contains one or more closed transverse walks; this can occur in one of two ways. If it occurs due to the presence of one or more transverse walks in $D$ from $\partial D$ to $\partial D$ and one or more transverse walks in $D'$ from $\partial D'$ to $\partial D'$, it can be avoided by removing those transverse walks from $D$ and $D'$. If it occurs as a transverse cycle consisting of the identified edges of $\partial D$ and $\partial D'$, then both $D$ and $D'$ are buffered, so it can be avoided by unbuffering ({\it i.e.}, deleting all boundary vertices and their incident edges) either $D$ or $D'$.

Figure \ref{fig:two_disks_construction} shows an example of two two disks constructions using the same disks $D$ and $D'$, but with different bijections. 
\begin{figure}
\begin{center}
\includegraphics[width=.97\textwidth]{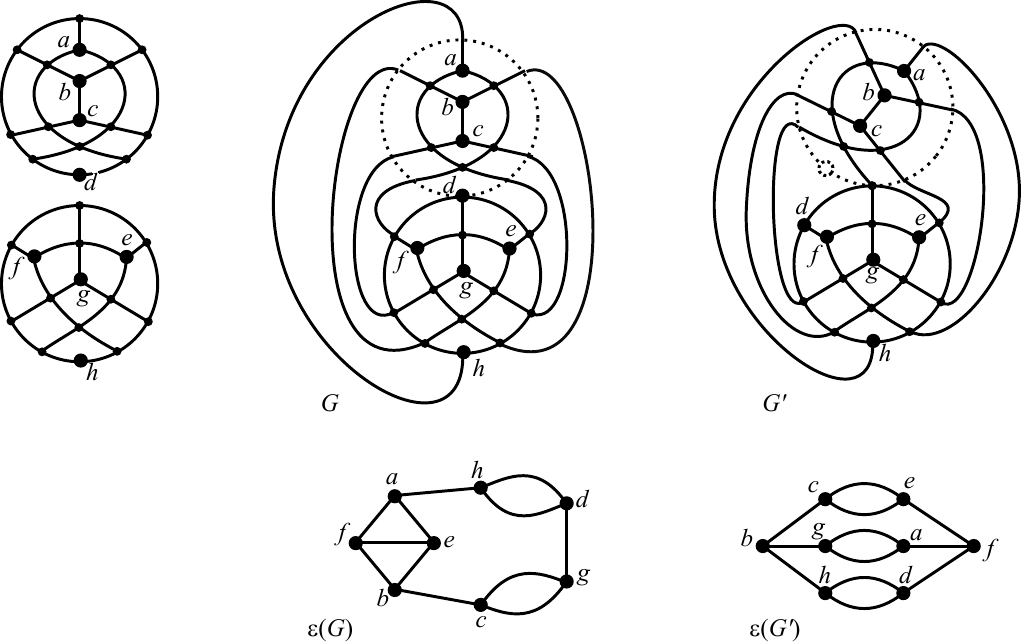}
\end{center}
\caption{On the right, quadrangulated immersions obtained from the two disks on the left, but via different bijections. Underneath, the extracted cubic graphs.}
\label{fig:two_disks_construction}
\end{figure}

\begin{prop} \label{pr:two_disk_conditions}
The two disk construction always produces a quadrangulated immersion. A quadrangulated immersion $G$ can be constructed via the two disks construction if and only if  $G$ contains a cycle comprised of complete transverse paths. 
\end{prop}
\begin{proof}
The first assertion is immediate since all faces, interior vertices and interior edges remain as they were in their respective disks,  the degree of each boundary vertex goes up by at most 1, and no vertices remain at degree 2. 

Suppose first that $G$ contains a cycle $C$ comprised of complete transverse paths. Since $C$ is a cycle it is a Jordan curve in the sphere, so is the common boundary of two closed disks $D$ and $D'$. Because $C$ is comprised of complete transverse paths, each of $D$ and $D'$ is quadrangulated. Corresponding to each degree 3 vertex on $C$ there will be a degree 2 vertex in exactly one of $\partial D$ or $\partial D'$; the corresponding vertex in the boundary of the other disk will have degree 3. Corresponding to each degree 4 vertex on $C$ there will be a degree 3 vertex in both $\partial D$ and $\partial D'$. 

Conversely, if the two disk construction is performed with disks $D$ and $D'$ such that  at least one of $\partial D$ and $\partial D'$ has a degree 2 vertex, then the image of $\partial D$ in $D \cup_\varphi D'$ will be a cycle comprised of complete transverse paths. 

Suppose then that neither $\partial D$ nor $\partial D'$ has a degree 2 vertex, so that the image of $\partial D$ in $D \cup_\varphi D'$ is a transverse cycle $C$. As shown in the case $i_3 = 4$ in the proof of Theorem \ref{thm:quadrangulated_disks}, $D$ is a buffering of a disk $\hat D$ for which $\partial \hat D$ does itself have a degree 2 vertex. The map $\varphi \colon \partial D \to \partial D'$ induces an obvious map $\hat\varphi\colon\partial \hat D \to \partial D'$, and the image of $\partial D$ in $\hat D \cup_{\hat\varphi} D' = D \cup_\varphi D'$ will be a cycle comprised of complete transverse paths. \end{proof}

\begin{figure}
\begin{center}
\includegraphics[width=.9\textwidth]{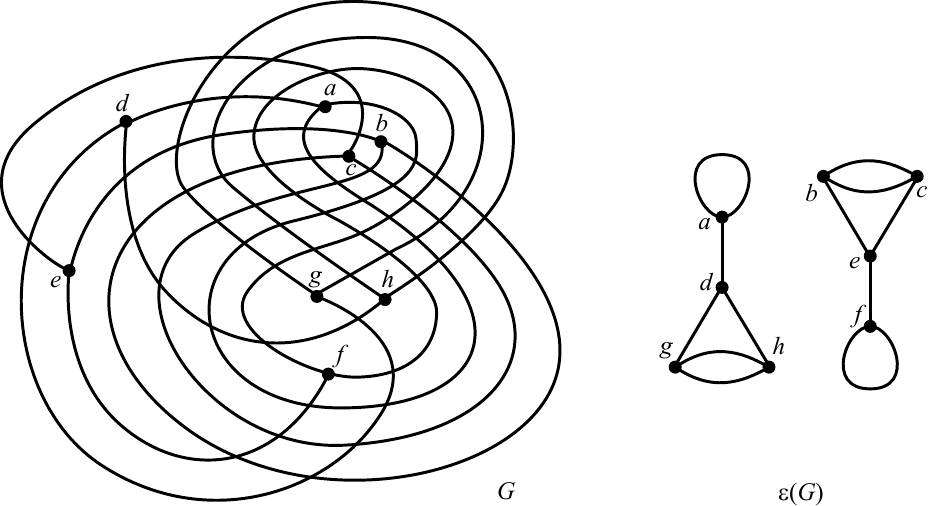}
\end{center}
\caption{A quadrangulated immersion, constructed by hand, that cannot be constructed using the two disks construction.}
\label{fig:not_two_disks}
\end{figure}
Given an edge $e$ in the extraction $\epsilon(G)$, we write $\epsilon^{-1}(e)$ to denote the complete transverse path in $G$ corresponding to $e$. The following result confirms that not all quadrangulated immersions can be constructed with the two disks construction. 
\begin{prop}
The example shown in Figure \ref{fig:not_two_disks} fails the condition in Proposition \ref{pr:two_disk_conditions}, and thus cannot be constructed with the two disks construction. 
\end{prop}
\begin{proof}
Let $G$ denote the quadrangulated immersion shown in Figure \ref{fig:not_two_disks}. We check the condition of Proposition \ref{pr:two_disk_conditions}.

Note first that each complete transverse path corresponds to an edge of $\epsilon(G)$, so to show that there are no cycles in $G$ comprised of complete transverse paths, we must show that no cycle in $\epsilon(G)$ can be realized as a cycle in $G$. 
\begin{itemize}
\item Let $aa$ and $ff$ denote the loops on vertices $a$ and $f$, respectively. Both $\epsilon^{-1}(aa)$ and $\epsilon^{-1}(ff)$  self-intersect, so neither is a cycle in $G$. 

\item The transverse paths $\epsilon^{-1}(dg)$ and $\epsilon^{-1}(dh)$ intersect each other, and thus cannot be part of a cycle in $G$. Similarly, the transverse paths $\epsilon^{-1}(ec)$ and $\epsilon^{-1}(eb)$ cannot be part of a cycle in $G$.

\item For one of the edges $\sigma$ connecting $g$ and $h$ in $\epsilon(G)$ the path $\epsilon^{-1}(\sigma)$ self-intersects, so we have no cycle in $G$ corresponding to the 2-cycle on $g$ and $h$ in $\epsilon(G)$. Similarly, for one of the edges $\tau$ connecting $b$ and $c$ in  $\epsilon(G)$ the path $\epsilon^{-1}(\tau)$ self-intersects. 
\end{itemize}
We see that there are no cycles in $G$ comprised of complete transverse paths.
\end{proof}
A second example that cannot be constructed with the two disks construction is shown in Figure \ref{fig:023}.

\subsection{The radial construction}

Given a cellular embedding $\Gamma$ of a graph $G$ in a surface $S$, the radial graph $R(G)$ is the graph with vertex set $V(\Gamma) \cup F(\Gamma)$ and edge set $\{vf \mid v\in V(\Gamma), f \in F(\Gamma), v \in \partial f\}$. $R(G)$ naturally inherits an embedding in $S$ from the embedding of $G$ in $S$, and it is easily verified that  the faces of $R(G)$ in $S$ are all quadrangles. A vertex of $R(G)$ corresponding to a vertex of $G$ has the same degree as in $G$, and a vertex of $R(G)$ corresponding to a face $f$ of $G$ has degree equal to the number of vertices on $\partial f$. It follows that if $G$ in $S$ is a quadrangulated immersion, then $R(G)$ in $S$ is also a quadrangulated immersion. See Figure \ref{fig:radial} for examples. 
\begin{figure}
\begin{center}
\includegraphics[width=.95\textwidth]{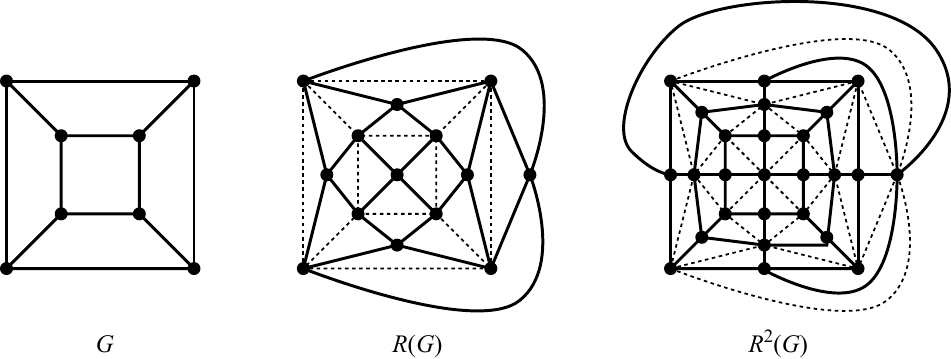}
\end{center}
\caption{A quadrangulated immersion, its radial graph, and the radial graph of its radial graph. In this example, $G$ is the cube  and $R(G)$ is a quadrangulated immersion of two copies of $K_4$.}
\label{fig:radial}
\end{figure}
Note that even if $G$ in $S$ was irreducible, it need not be the case that $R(G)$ in $S$ is irreducible. In Figure \ref{fig:radial}, $R(G)$ is irreducible but $R^2(G)$ is not. In fact, in that example we have $\epsilon(R^2(G)) = G$; this is a special case of a general property, as Proposition \ref{pr:dualradial} states.

\begin{prop} \label{pr:dualradial}
For any embedding of a graph $G$ in a surface $S$, we have $$\epsilon\left(R^2(G)\right) = \epsilon(G).$$ 
\end{prop}
\begin{proof}
Observe that, from any embedded graph $G$, the double radial $R^2(G)$ can be obtained by
\begin{enumerate} 
\item subdividing each edge $e$ of $G$ with a new vertex $v_e$ in its center; 
\item adding a vertex $v_f$ in the center of each face $f$ of $G$; and,
\item for each face $f$ and each edge $e$ on $\partial f$, adding the edge $v_ev_f$. 
\end{enumerate}
It follows, therefore, that each edge-vertex $v_e$ and each face-vertex $v_f$ has degree 4, and all new edges belong to the induced subgraph on the set of such vertices. This essentially says that $R^2(G)$ is obtained from $G$ by adding in a family of transversals.
\end{proof}

For the next result, note that if $G$ is a quadrangulated immersion then vertices in $R(G)$ corresponding to faces of the embedding must have degree 4. It follows that if $u$ is a vertex of $\epsilon(R(G))$, then $u$ is a vertex of $R(G)$ corresponding to a vertex of $G$. 

\begin{theorem} 
If $G$ is a quadrangulated immersion with an odd quantity of vertices, then $\epsilon(R(G))$ is not connected. Furthermore, if $\epsilon(R(G))$ is connected then the induced embedding of $R(G)$  contains at least one transversal. 
\end{theorem}
\begin{proof}
Observe first that if $G$ is a quadrangulated immersion and $u,w$ are vertices in the same component of $\epsilon(R(G))$, then $u,w$ correspond to vertices in the same class of the bipartition of $G$. We can see this as follows: Suppose $u,w$ are degree-3 vertices in $R(G)$ and $P$ is a transverse path from $u$ to $w$ in $R(G)$. Write $P=v_1,f_1,v_2,f_2, \ldots, v_k,f_k$ where $v_i$ are vertices of $R(G)$ corresponding to vertices of $G$ (with $v_1=u$ and $v_k=w$) and $f_i$ are vertices of $R(G)$ corresponding to faces of the embedding of $G$. Since every face of $G$ is a quadrangle we can construct an even-length path $P'$ in $G$ from $u$ to $w$ by replacing each subpath $v_i,f_i,v_{i+1}$ with one of the two subpaths $v_i,x_i,v_{i+1}$, where $x_i$ is a vertex on $\partial f_i$ other than $v_i$ and $v_{i+1}$.  

If $G$ is a quadrangulated immersion with an odd quantity of vertices, then by Corollary \ref{cor:6and2} not all degree 3 vertices are in the same bipartition class of $G$. Combining this with the observation from the previous paragraph, we see that  $\epsilon(R(G))$ cannot be connected. 

On the other hand, if $\epsilon(R(G))$ is connected then all cubic vertices of $G$ must be in the same class of the bipartition of $G$. Thus all vertices in the other bipartition class of $G$ have degree 4, so that one transverse-path component of $R(G)$ consists only of vertices of degree 4, and therefore contains at least one transversal. 
\end{proof}

Figure \ref{fig:bipartite_radial} illustrates an example of a quadrangulated immersion $G$ for which $\epsilon(R(G))$ is not connected, and Figure \ref{fig:bipartite_radial_connected} illustrates an example of a quadrangulated immersion $G$ for which $\epsilon(R(G))$ is connected. 
\begin{figure}
\begin{center}
\includegraphics[width=.97\textwidth]{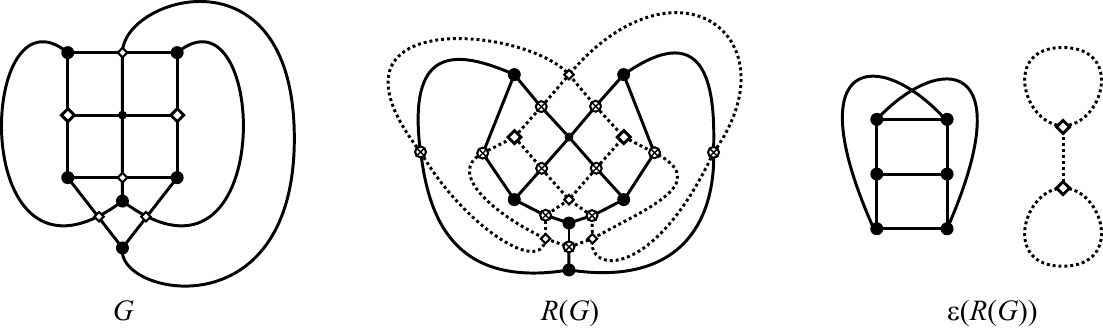}
\end{center}
\caption{On the left, a quadrangulated immersion $G$ with the vertices of one bipartition class shown filled in and the vertices of the other bipartition class shown as diamonds with white centers. In the center, the radial graph $R(G)$, with vertices corresponding to faces of $G$ shown as ``$\otimes$'' and the distinction between rendering of edges indicating which bipartition class they belong to. On the right, the disconnected extraction $\epsilon(R(G))$.}
\label{fig:bipartite_radial}
\end{figure}
\begin{figure}
\begin{center}
\includegraphics[width=.97\textwidth]{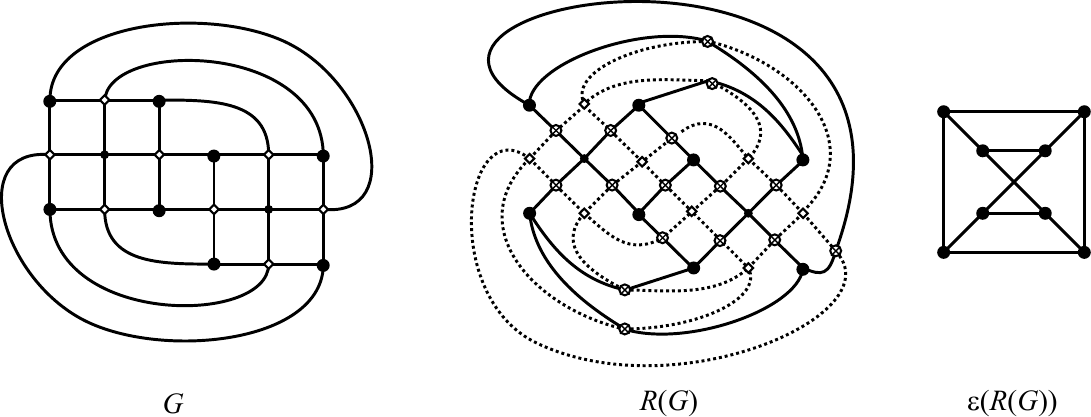}
\end{center}
\caption{On the left, a quadrangulated immersion $G$ with the vertices of one bipartition class shown filled in and the vertices of the other bipartition class shown as diamonds with white centers. In the center, the radial graph $R(G)$, with vertices corresponding to faces of $G$ shown as ``$\otimes$'' and the distinction between rendering of edges indicating which bipartion class they belong to. Note that the dashed edges and diamond vertices comprise a single transversal. On the right, the connected extraction $\epsilon(R(G))$.}
\label{fig:bipartite_radial_connected}
\end{figure}

\subsection{The spiral construction}

Let disk $D$ have a quadrangulated immersion with $n\geq 6$ vertices on $\partial D$, including at least one degree-2 vertex and one degree-3 vertex on $\partial D$. Label the vertices on $\partial D$ consecutively with $v_0, v_1, \ldots, v_{n-1}$ where $v_0$ has degree 3 and $v_{n-1}$ has degree 2 and choose an integer $l\geq 1$. If $l<n/2$ then we impose the additional requirement that $v_{l+1},v_{l+n/2},$ and $v_{l+n/2+1}$ all have degree 3. The {\it spiral construction} proceeds in two steps. 
\begin{figure}
\begin{center}
\includegraphics[width=.95\textwidth]{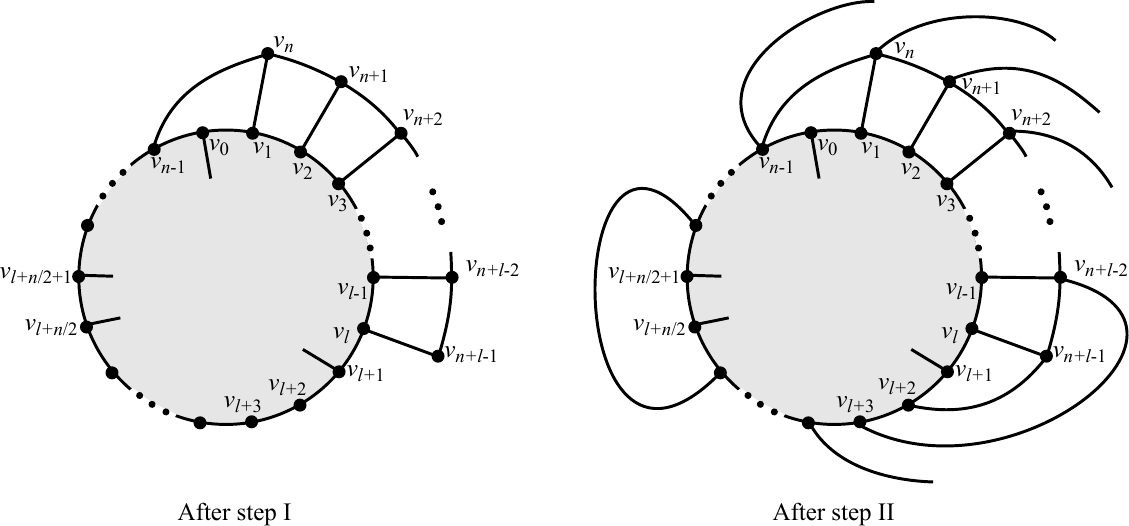}
\end{center}
\caption{The spiral construction applied with $l<n/2$.}
\label{fig:spiral-generic}
\end{figure}
\begin{figure}
\begin{center}
\includegraphics[width=.95\textwidth]{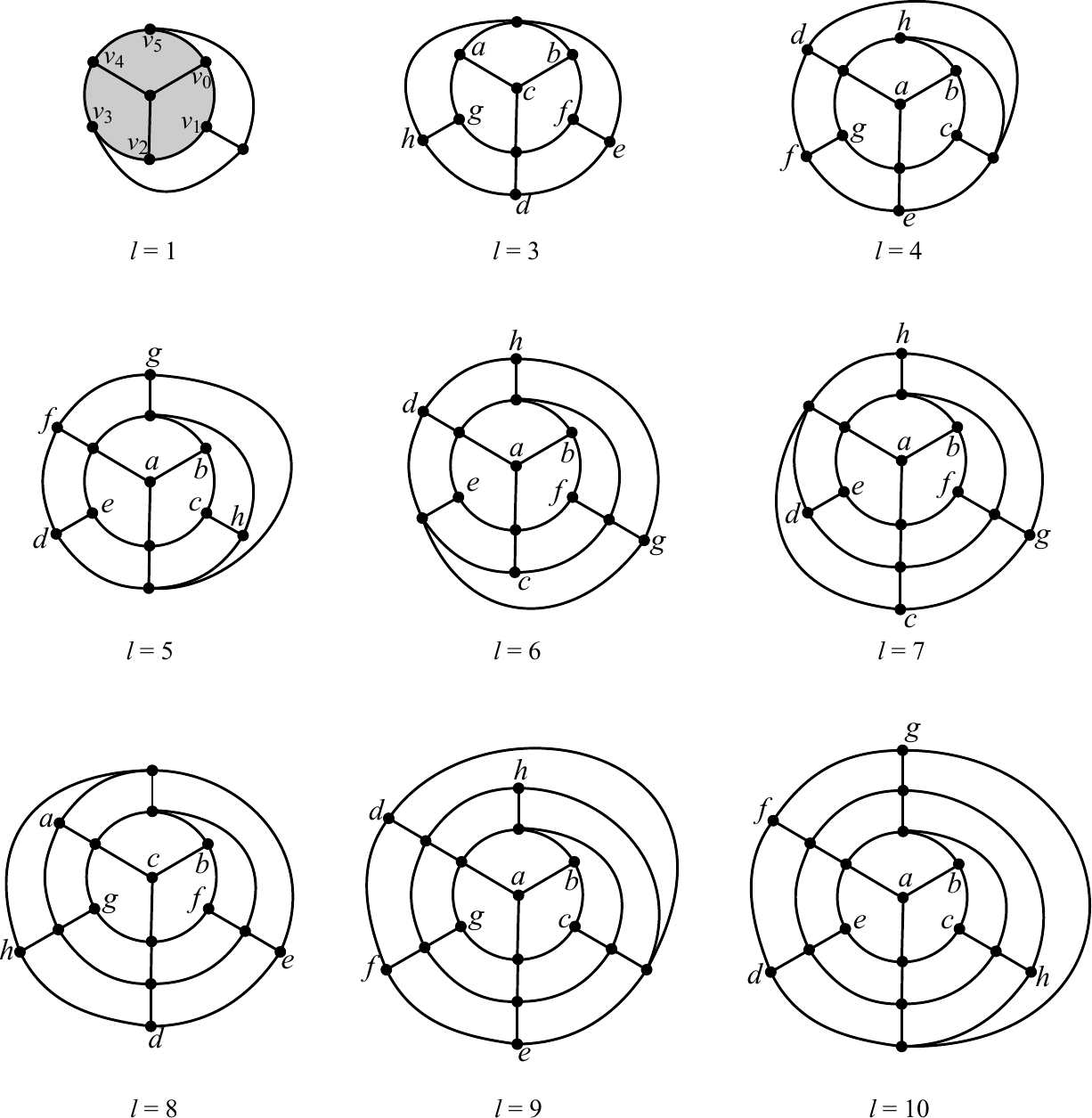}
\end{center}
\caption{An example of applying the spiral construction for various values of $l$. The original disk, with the vertex labeling used in the construction, is shown for the case $l=1$. The vertex labelings for cases $l\geq 3$ are provided for comparison with with Figure \ref{fig:periodic_extractions}. Note that the case $l=2$ violates one of the requirements for the construction.}
\label{fig:spiral-example}
\end{figure}

\medskip
\noindent {\bf Step I.\ } For $i=0, 1, \ldots, l-1$ add a new vertex $v_{n+i}$ and new edges $v_{i+1}v_{n+i}$ and $v_{n+i-1}v_{n+i}$. 

\medskip
\noindent {\bf Step II.\ } For $j=0,1, \ldots, n/2-3$ add edges $v_{l+n-1-j}v_{l+2+j}$. 

\medskip
Figure \ref{fig:spiral-generic} depicts the spiral construction when $l<n/2$, and Figure \ref{fig:spiral-example} shows a specific instance of the spiral construction. Given a quadrangulated immersion $\Gamma$ of a disk $D$ with a valid choice of cyclic labeling of the vertices on $\partial D$, for each valid choice of $l$ we write $S_l(\Gamma)$ for the spherical embedding resulting from applying the spiral construction with the given choices. 

\begin{prop} \label{pr:spiral_is_valid}
For any quadrangulated immersion $\Gamma$ of a disk $D$ with a valid choice of cyclic labeling of the vertices on $\partial D$ and any valid choice of $l>0$, the embedding $S_l(\Gamma)$ is a quadrangulated immersion. 
\end{prop}
\begin{proof}
We first analyze Step I. We begin with all $n$ vertices on $\partial D$ of degree 2 or 3. With each increment of $i$, a new quadrangle is added to the exterior of the disk, increasing the degree of $v_{i+1}$ and $v_{n+i-1}$ by 1, so each is now of degree 3 or 4, whereas the new vertex $v_{n+i}$ has degree 2. In the initial case $i=0$ the degree of $v_0$ has not been changed, which is precisely why the vertex labeling must be chosen so that $v_0$ begins with degree 3. Note that, for each increment of $i$, no interior vertices are modified, so all are of degree 3 or 4, and all vertices on the boundary are of degree 2 or 3, with the possible exception of $v_i$, which may have degree 4. Similarly, no existing interior faces are modified, so all interior faces are quadrangles. At the end of Step I the boundary still has length $n$, as the loss of vertex $v_0$ is balanced by the gain of vertex $v_{n+l-1}$. 

We now analyze Step 2. For each increment of $j$, the edge $v_{l+n-1-j}v_{l+2+j}$ creates a new quadrangle. Because $v_{l+n-1-j}$ and $v_{l+2+j}$ were on the boundary of the disk at the end of Step 1 and neither is $v_l$, they each had degree 2 or 3 previously, and thus each has degree 3 or 4 now. The vertex $v_l$ is now interior, so it is not an issue that it might have degree 4. In the initial case $j=0$ the vertex $v_{l+1}$ is incorporated into the interior via the addition of the new face without increasing its degree. This is not an issue, since if $l<n/2$ we require that it have degree 3, and if $l\geq n/2$ it is guaranteed  that $v_{l+1}$ will have degree 3 at the conclusion of Step 1. 

At each increment of  $j$, the boundary of the disk becomes shorter by 2 edges. Since our graph is bipartite the original boundary length, $n$, is even, and thus Step 2 continues until the boundary has only vertices  $$v_{l+n/2-1},\, v_{l+n/2},\, v_{l+n/2+1},\, v_{l+n/2+2}.$$ Because of the final increment of $j$, {\it i.e.}, $j=n/2 - 3$, the vertices  $v_{l+n/2-1}$ and $v_{l+n/2+2}$ each have degree 3 or 4. The degrees of $v_{l+n/2}$ and $v_{l+n/2+1}$ are not changed in the course of Step II, but again this is not an issue since  if $l<n/2$ we require that they each have degree 3, and if $l\geq n/2$ it is guaranteed  that they will have degree 3 at the conclusion of Step 1. 
\end{proof}
In light of Proposition \ref{pr:spiral_is_valid} and Proposition \ref{pr:one_crossing}, the example in Figure \ref{fig:spiral-example} verifies Corollary \ref{cor:all_counts}.
\begin{corollary} \label{cor:all_counts}
For every $m\in\N_{\geq 0}$ with $m \neq 1$ , there is a quadrangulated immersion with $8+m$ vertices.  
\end{corollary}

Although the sequence $S_l(\Gamma)$ is infinite, it only yields quadrangulated immersions of finitely many cubic graphs. 
\begin{prop} \label{pr:spiralperiodicity}
For any quadrangulated immersion $\Gamma$ of a disk $D$ with a valid choice of cyclic labeling of the $n$ vertices on $\partial D$ and any valid choice of $l>1$, we have $\epsilon(S_l(\Gamma)) = \epsilon(S_{l+n-1}(\Gamma))$.
\end{prop}
\begin{figure}
\begin{center}
\includegraphics{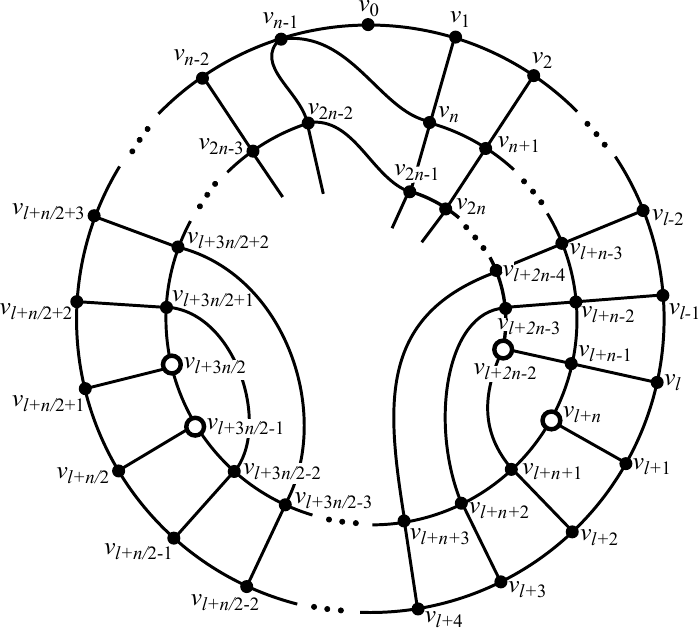}
\end{center}
\caption{The result of applying Steps I and II to obtain $S_{l+n-1}(\Gamma)$; the interior of the original disk, which occupies the outer face, is not shown. The degree 3 vertices outside the original disk are shown enlarged with white centers. }
\label{fig:periodicity_proof}
\end{figure}
\begin{proof}
Since this is a periodicity result, it suffices to prove it for values $l < n$. We will define a graph isomorphism $\varphi\colon \epsilon(S_l(\Gamma)) \to \epsilon(S_{l+n-1}(\Gamma))$; for ease of notation we will write $u_i$ for vertices of $S_l(\Gamma)$ and $v_i$ for vertices of $S_{l+n-1}(\Gamma)$, understanding that for some values of $i$ one can reasonably assert that $u_i=v_i$. 

The interior of the original disk $D$ contains some number of degree 3 vertices,  and any vertices on $\partial D$ which originally had degree 2 now have degree 3 in both $\epsilon(S_l(\Gamma))$ and $S_{l+n-1}(\Gamma)$; on all these vertices we define $\varphi$ to be the identity. Additionally, we define $\varphi$ to map  the degree 3 vertices 
$$
u_0, u_{l+n/2}, u_{l+n/2+1}, u_{l+1}, \mbox{ and } u_{l+n-1},
$$
in  $\epsilon(S_l(\Gamma))$ to the degree 3 vertices
$$
v_0, v_{l+3n/2 - 1}, v_{l+3n/2}, v_{l+n}, \mbox{ and } v_{l+2n-2},
$$
in  $\epsilon(S_{l+n-1}(\Gamma))$, respectively. 

From Figure \ref{fig:periodicity_proof} it is not difficult to see that there is now 
 a correspondence of complete transversal paths in $S_l(\Gamma)$ with complete transversal paths in $S_{l+n-1}(\Gamma)$ which is consistent with $\varphi$. In particular, we have the following correspondences:
 
 \medskip
 \begin{tabular}{rcl}
 the complete transversal  &  &  the complete transversal  \\ 
path in $S_l(\Gamma)$ containing & & path in $S_{l+n-1}(\Gamma)$ containing \\\cline{1-1}\cline{3-3} \\
 
 vertex  $u_{i+1}$ but no edge in $\partial D$ & $\leftrightarrow$ & edge $v_{i+1}v_{n+i}$ (for $0 \leq i < l-1$)\\[5pt]

edge $u_{n+i-1}u_{n+i}$ & $\leftrightarrow$ & edge $v_{2n-2+i}v_{2n-1+i}$  (for $0 \leq i < l$)\\[5pt]
edge $u_{l+n-1-j}u_{l+2+j}$   & $\leftrightarrow$ &  path \\
& & $v_{l+n-1-j}v_{l+2n-2-j}v_{l+n+1+j}v_{l+2+j}$\\
& & (for $0 \leq j \leq n/2-3$)
\end{tabular}
\end{proof}

Figure \ref{fig:periodic_extractions} shows the extracted graphs obtained from the spiral construction family shown in Figure \ref{fig:spiral-example}.
\begin{figure}
\begin{center}
\includegraphics[width=.8\textwidth]{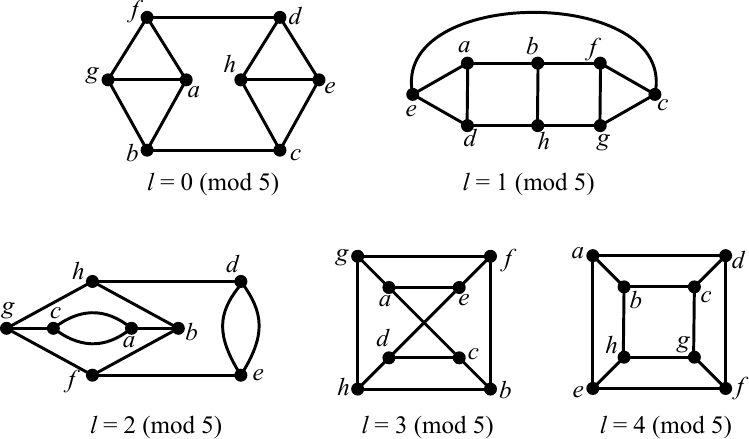}
\end{center}
\caption{The extractions of the examples shown in Figure \ref{fig:spiral-example}, with corresponding vertex labelings.}
\label{fig:periodic_extractions}
\end{figure}

\subsection{The cable construction}

In this section we present a construction that is similar in flavor to the spiral construction, but exhibits somewhat more flexibility. 

Suppose $G$ is a quadrangulated immersion, $W^*$ is a  walk in the dual embedding $G^*$, and $e$ is an edge of $G$ corresponding to a dual edge $e^*$ occurring in $W^*$. Let $f$ denote the face the walk $W^*$ arrives at after crossing $e^*$, and label the edges of $\partial f$ as shown in Figure \ref{fig:turns}.
\begin{figure}
\begin{center}
\includegraphics{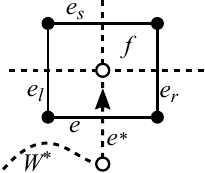}
\end{center}
\caption{Labeling the edges of $\partial f$.}\label{fig:turns}
\end{figure}
We say that $W^*$ {\it turns to the right after edge $e$} if after crossing $e$ the next edge in $W^*$  is the dual edge $e_r^*$. Similarly we say that $W^*$ {\it continues straight after $e$} or {\it turns left after $e$} if the next edge in $W^*$ is the dual edge $e_s^*$ or $e_l^*$, respectively.  

Let $G$ be a quadrangulated immersion, and let  $W^*$ be a walk in $G^*$ having $2k$ turns which alternate right and left and further require that if $v^*$ is a vertex at which $W^*$ has four edges, then $W^*$ has two turns at $v^*$. We refer to a walk satisfying these conditions as a {\it cabling walk.} Number the edges of $W^*$ in order as $e_1^*, e_2^*, \ldots, e_n^*$ (where indices are taken modulo $n$), and let $r_1 < l_1 < r_2 < l_2 < \cdots < r_k < l_k$ be indices such that, for each $i$, $W^*$ turns to the right after crossing edge $e^*_{r_i}$, and then turns to the left after crossing edge $e^*_{l_i}$. For each dual edge $e_i^*$ we think of the edge $e_i$  in $G$ as oriented left to right relative to the direction (thought of as ``straight'') that $W^*$ traverses $e_i$, and write $u_{i,1},u_{i,2}$ for the source of $e_i$ and target of $e_i$, respectively. Let 
$$I_R := \{r_1, r_2, \ldots, r_k\}, \ \ I_L:= \{ l_1, l_2, \ldots, l_k\},$$
and
$$ I := \bigcup_{i=1}^k\{r_i+1,r_i+2, \ldots, l_i-1\}.$$ 

Given a number $c \geq 1$, we apply the following {\it cable construction} to $G$ to obtain a new quadrangulated immersion $\C(G,W^*,c)$. We describe each step in such a way that it is clear not only what happens combinatorially, but also how the vertices and edges should be placed topologically. In particular, no intersections are introduced except at new vertices, as specified. 
\begin{enumerate}
\item \label{step:vert_I-IL} For each  $i\in I\cup I_L$ add $c-1$ new vertices $v_{i,1}, v_{i,2}, \ldots, v_{i,c-1}$ to the edge $e_i$, arranged from left to right. 

\item \label{step:vert_notI-IL} For $i\not\in I\cup I_L$ add $c$ new vertices $v_{i,1}, v_{i,2}, \ldots, v_{i,c}$ to the edge $e_i$, arranged from left to right.

\item \label{step:edge_notI-IR-IL} For  $i \not\in I \cup I_R \cup I_L$ and each $j=1, 2, \ldots, c$ add an edge $v_{i,j}v_{i+1,j}$. 

\item \label{step:edge_I} For  $i \in I$ add the edge $v_{i,1}u_{i+1,1}$, the edges $v_{i,j}v_{i+1,j-1}$ for  $j=2, \ldots, c-1$, and the edge $u_{i,2}v_{i+1,c-1}$.

\item \label{step:edge_IR} For each $r\in I_R$  add the   edge $v_{r,1}u_{r+1,1}$ and the edges $v_{r,j}v_{r+1,j-1}$ for  $j=2, \ldots, c$.

\item \label{step:edge_IL} For each $l\in I_L$ add the edges $v_{l,j}v_{l+1,j}$ for $j=1, \ldots, c-1$ and the edge $u_{l,2}v_{l+1,c}$.

\item \label{step:smooth} For each  $i\in I\cup I_L$ delete the edges $u_{i,1}v_{i,1}$, $v_{i,j}v_{i,j+1}$ for $j=1, 2, \ldots, c-2$, and also $v_{i,c}u_{i,2}$, and then smooth all vertices now having degree 2.
\end{enumerate}

For convenience, we define $\C(G,W^*,0)$ to be $G$ itself. See Figure \ref{fig:cabling_example} for examples of the cabling construction.
\begin{figure}
\begin{center}
\includegraphics[width=.97\textwidth]{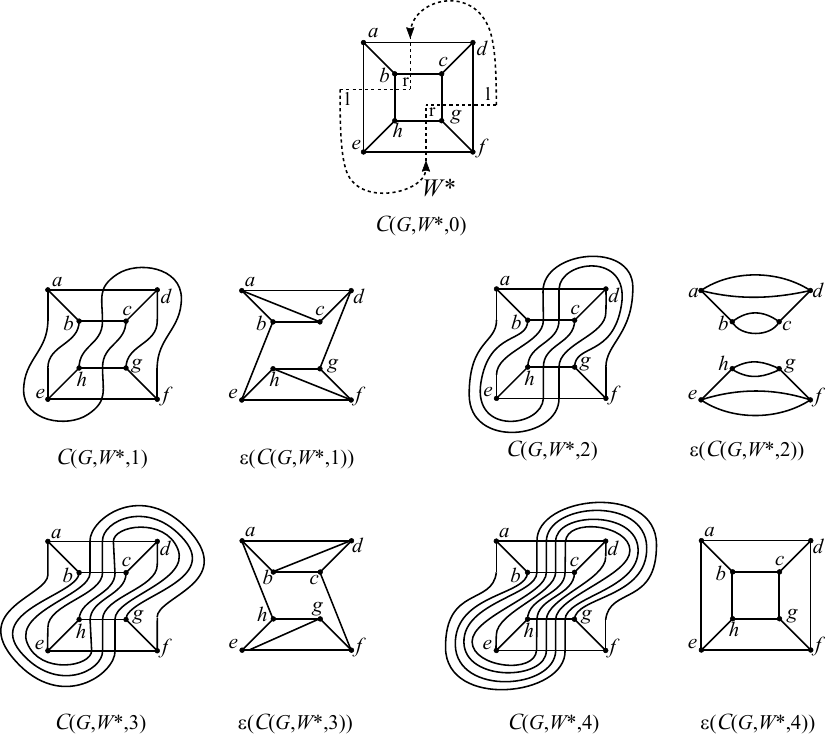}
\end{center}
\caption{On top, a quadrangulated immersion $G$ with a choice of cabling walk $W^*$. Below are examples of the cabling construction for $c=1,2,3,4$, along with the corresponding extracted cubic graph.}
\label{fig:cabling_example}
\end{figure}

\begin{theorem}\label{thm:cable}
Suppose $G$ is a quadrangulated immersion and  $W^*$ is a cabling walk in $G^*$. For any natural number $c$ the embedding $\C(G,W^*,c)$ is a quadrangulated immersion, and for any $c$ we have 
$$\epsilon\left(\C(G,W^*,c+|I_R\cup I|)\strut\right) = \epsilon\left(\C(G,W^*,c)\strut\right) .$$
\end{theorem}
\begin{proof}
To show that $\C(G,W^*,c)$ is a quadrangulated immersion, is suffices to analyze only those part of $G$ that were changed. 

First, we confirm that all vertices are of degree 3 or 4. For $i \in I\cup I_L$ the vertices $v_{i,j}$ do not actually appear in   $\C(G,W^*,c)$, since they have been smoothed (Step \ref{step:smooth}). For $i\not\in I \cup I_L$, all new vertices $v_{i,j}$ have degree 4, as two incident edges arose from the subdivision of $e_i$ (Step \ref{step:vert_notI-IL}) and the additional two incident edges are $v_{i,j}v_{i+1,j}$ and $v_{i-1,j}v_{i,j}$ for $i\not\in I_R$ (Step \ref{step:edge_notI-IR-IL}), and for $i\in I_R$ they are $v_{i,j}v_{i+1,j-1}$ and $v_{i-1,j}v_{i,j}$ for $j\neq 1$ and  $v_{i,1}u_{i+1,1}$ and $v_{i-1,1}v_{i,1}$ for $j= 1$ (Step \ref{step:edge_IR}). Finally, for $i \in I\cup I_L$ the vertices $u_{i,1}$ and $u_{i,2}$ have the same degree in  $\C(G,W^*,c)$  as in $G$, because the edges $u_{i,1}v_{i,1}$ and $v_{i,c-1}u_{i,2}$ were deleted (Step \ref{step:smooth}) but edges $v_{i,1}u_{i+1,1}$ for $i\in I_R \cup I$ and edges $u_{i,2}v_{i+1,c}$ for $i\in I\cup I_L$ were added (Steps \ref{step:edge_I}, \ref{step:edge_IR}, \ref{step:edge_IL}).

Now we confirm that all faces are quadrangles. Write $f_i$ for the face having edges $e_i$ and $e_{i+1}$ in its boundary; since $W^*$ is a closed walk, these are exactly the faces changed in going from $G$ to $\C(G,W^*,c)$. Prior to Step \ref{step:smooth}, the face $f_i$ has been subdivided, left to right, into 
\begin{itemize}
\item $c+1$ new quandrangles, for $i\not\in I_R \cup I \cup I_L$ (Step \ref{step:edge_notI-IR-IL}). 
\item  a triangle, $c-1$ quadrangles, and another triangle, for $i \in I$ (Step \ref{step:edge_I});
\item $c$ quadrangles and a triangle, for $i\in I_R$ (Step \ref{step:edge_IR});
\item a triangle and $c$ quadrangles, for $i \in I_L$ (Step \ref{step:edge_IL}).
 \end{itemize}
In Step \ref{step:smooth}, one edge of each of the triangles mentioned above is deleted; with the smoothing of degree 2 vertices each becomes part of a quadrangle; see Figure \ref{fig:cabling_quadrangles}.
\begin{figure}
\begin{center}
\includegraphics{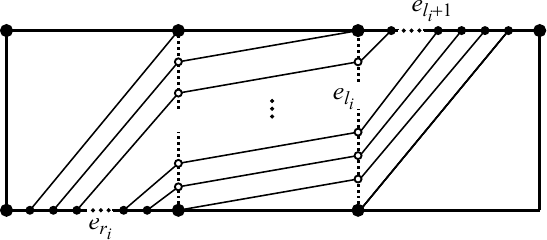}
\end{center}
\caption{An example showing the quadrangles created in faces $f_{r_i}$, $f_{r_i+1}$, \ldots, $f_{l_i}$. In Step \ref{step:smooth}, dashed edges are deleted and then white vertices are smoothed. }
\label{fig:cabling_quadrangles}
\end{figure}

Finally, to prove the statement about the extracted graphs, we study the complete transverse paths in $\C(G,W^*,c)$. Note first that the vertices $u_{i,1}$ and $u_{i,2}$, for $i=1, 2, \ldots, n$ are the only vertices of $G$ affected by the construction. For each $i\not\in I\cup I_L$, the edge $e_i$ from $u_{i,1}$ to $u_{i,2}$ in $G$ is replaced by the transverse path $u_{i,1}v_{i,1}v_{i,2}\cdots v_{i,c}u_{i,2}$, which does not affect the structure of the extracted graph regardless of the value of $c$.

For $i\in I\cup I_L$, prior to performing Step \ref{step:smooth} there is a transverse path $P_i = u_{i,1}v_{i-1,a_1}v_{i-2,a_2}\cdots v_{i-m+1,a_{m-1}} u_{i-m,2}$ where the sequence $\{a_j\}$ is defined by 
\begin{itemize}
\item $a_1 = 1$; 
\item for each $j$,   $a_{j+1}=a_j$ if $i-j-1 \not\in I_R\cup I$ whereas $a_{j+1} = a_j+1$ if $i-j-1 \in I_R \cup I$; and
\item $a_{m-1} = c-1$.
\end{itemize}
Furthermore, we necessarily have either $i-m+1 \in I \cup I_L$ or $i-m\in I_L$, since the last edge has $u_{i-m,2}$ as its second end. After Step 7, some of the intermediate vertices in $P_i$ will be gone, but $P_i$ will still be a transverse path from $u_{i,1}$ to $u_{i-m,2}$. Now note that, in $\C(G,W^*,c+|I_R\cup I|)$, the transverse path $P_i$ will be replaced by a transverse path $P'_i$ from $u_{i,1}$ to $u_{i-m',2}$ where $a_{m'-1} = c+ |I_R\cup I| - 1$. Since we have $a_{j+1} = a_j+1$ exactly when $i-j-1 \in I_R \cup I$, we see that $P'_i$ does a full traversal of all edges of $I_R \cup I$, and hence all $n$ edges of $W^*$, in addition to what occurred in $P_i$. It follows that $i-m' = i - m - n = i-m \mod n$, so $u_{i-m',2}=u_{i-m,2}$. 

We see that the transverse paths in $\C(G,W^*,c+|I_R\cup I|)$ connect the degree-3 vertices as in $\C(G,W^*,c)$, so we are done.
\end{proof}
Since the cabling walk in Figure \ref{fig:cabling_example} has $|I_R \cup I| = 4$, Theorem \ref{thm:cable}  tells us that the seqence of extracted cubic graphs depicted there repeats periodically according to the value of $c \mod 4$. 

Finally, we observe that Proposition \ref{pr:spiralperiodicity} immediately follows as a corollary of the last part of Theorem \ref{thm:cable}  when, in the notation of Proposition \ref{pr:spiralperiodicity}, we have $\ell\geq n$, but otherwise not.

\section{Computations} \label{sec:computations}

\subsection{Cubic multigraphs on 8 vertices} \label{subsec:cubic_multigraphs}
For the sake of study and cataloging of results, all cubic multigraphs on 8 vertices were generated. The count of connected cubic multigraphs on up to 24 vertices is given by Brinkmann {\it et al.} (they refer to these as connected cubic multigraphs with loops) \cite{BrinkmannEtAl13}. Here, we also provide drawings of all these graphs on up to 8 vertices. 

Computer search reaffirmed that there are exactly 71 connected cubic multigraphs on 8 vertices. Of these, the four which are 3-connected are shown in Figure \ref{fig:census_3connected}. All four of these graphs are simple; the fifth simple graph appears in the list below as, in the notation defined below, 4d$\myd$4d.

\begin{figure}
\begin{center}
\includegraphics[width=.85\textwidth]{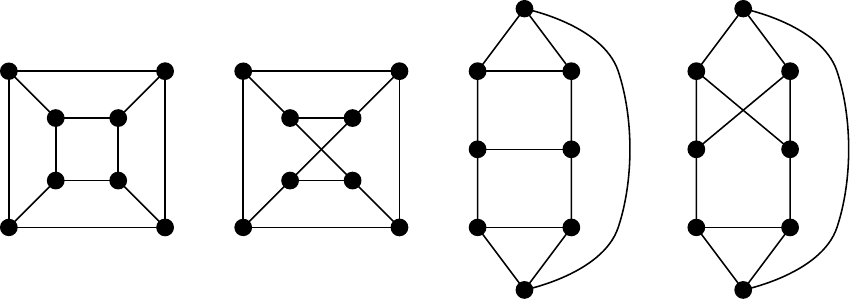}
\end{center}
\caption{The 3-connected cubic graphs on 8 vertices.}
\label{fig:census_3connected}
\vspace{.5in}
\end{figure}

To list the non-3-connected examples we use various cubic multigraphs on 2 to 6 vertices, as shown in Figures \ref{fig:pieces_2} -- \ref{fig:pieces_6}. These multigraphs are labeled in the form $nx$ where $n$ indicates the number of vertices and $x$ is an arbitrary alphabetic designator. Given  multigraphs $G,H$ with  a single half-edges each, we write $G\mys H$ to denote the result of attaching the half-edge of $G$ to the half-edge of $H$.  Given  multigraphs $G,H$ with 2 half-edges each, we write $G\myd H$ to denote the result of attaching the half-edges of $G$ to the half edges of $H$, respectively. How to pair the half-edges of $G$ with those of $H$ will not matter for us, as in all cases we list, one of $G$ or $H$ will be symmetric with respect to its half-edges. When $G_1$ and $G_2$ each have a single half-edge and $H$ has 2, we write $G_1\mys H\mys G_2$ to denote the result of the obvious attachments. When $G_1$ and $G_2$ each have two half-edges, we write $G_1\myd 4g \myd G_2$ to denote the result of attaching the two half-edges of $G_1$ to the left half-edges of multigraph 4g and attaching the half-edges of $G_1$ to the right half-edges of 4g. Brinkmann {\it et al.} identify various construction operations that are equivalent or analogous to our action of attaching \cite{BrinkmannEtAl13}.

The following is a complete list of representatives of the distinct isomorphism classes of all connected, non-3-connected cubic multigraphs on 8 vertices. 
\begin{itemize}
\item All multigraphs of the form $G\myd H$, where $G,H\in\{$4a, 4b, 4c, 4d, 4e, 4f$\}$, excluding the case 4f\myd 4f. In fact, there are two ways to connect 4f to itself. One is listed as 4a\myd 4e and the other is listed below as 6d\myd 2a. Note that 4d\myd 4d is the fifth simple connected cubic graph on 8 vertices. \smallskip 

\item All multigraphs of the form $G\myd H$ where $G\in\{$6a, 6b, 6c, 6d, 6e$\}$ and $H\in\{$2a, 2b, 2c$\}$, excluding 6a\myd 2a and 6f\myd 2a (which are listed above as 4c\myd 4c and 4c\myd 4f, respectively.)\smallskip

\item All multigraphs of the form $G_1\myd 4\mbox{g}\myd G_2$ where $G_1,G_2\in\{$2a, 2b, 2c$\}$, excluding 2a\myd 4g\myd 2a and 2b\myd 4g\myd 2b (which are listed above as 4e\myd 4e and 6b\myd 2a, respectively.)\smallskip

\item All multigraphs of the form $G\mys H$ where $G\in\{$5a, 5b, 5c, 5d, 5e, 5f$\}$ and $H\in\{$3a, 3b, 3c$\}$. \smallskip

\item All multigraphs of the form $G_1\mys H\mys G_2$ where $G_1\in\{$3a, 3b, 3c$\}$ and $H\in\{$2b, 2c$\}$. 
\end{itemize}

\begin{figure}
\begin{center}
\includegraphics[width=.55\textwidth]{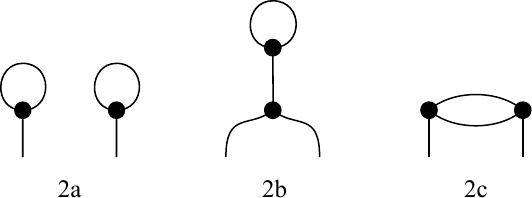}
\end{center}
\caption{Some cubic multigraphs on 2 vertices, with half-edges.}
\label{fig:pieces_2}
\end{figure}

\begin{figure}
\begin{center}
\includegraphics[width=.7\textwidth]{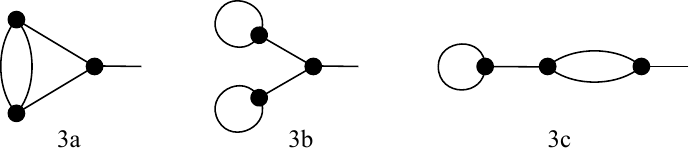}
\end{center}
\caption{Some cubic multigraphs on 3 vertices, with half-edges.}
\label{fig:pieces_3}
\end{figure}

\begin{figure}
\begin{center}
\includegraphics[width=.75\textwidth]{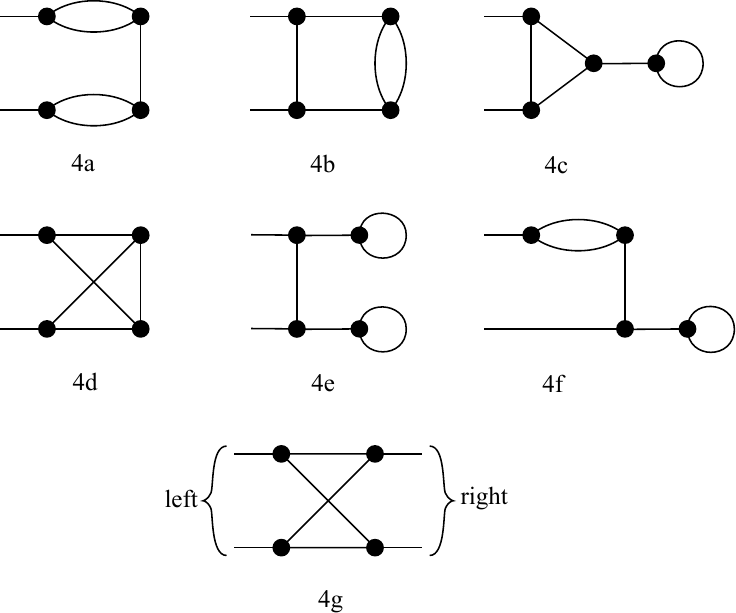}
\end{center}
\caption{Some cubic multigraphs on 4 vertices, with half-edges.}
\label{fig:pieces_4}
\end{figure}

\begin{figure}
\begin{center}
\includegraphics[width=.8\textwidth]{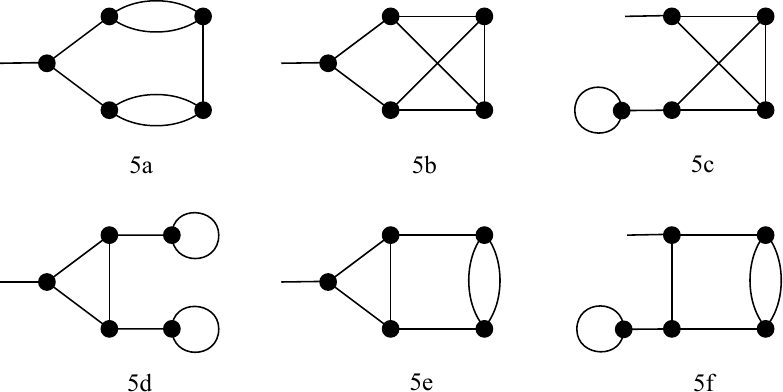}
\end{center}
\caption{Some cubic multigraphs on 5 vertices, with half-edges.}
\label{fig:pieces_5}
\end{figure}

\begin{figure}
\begin{center}
\includegraphics[width=.75\textwidth]{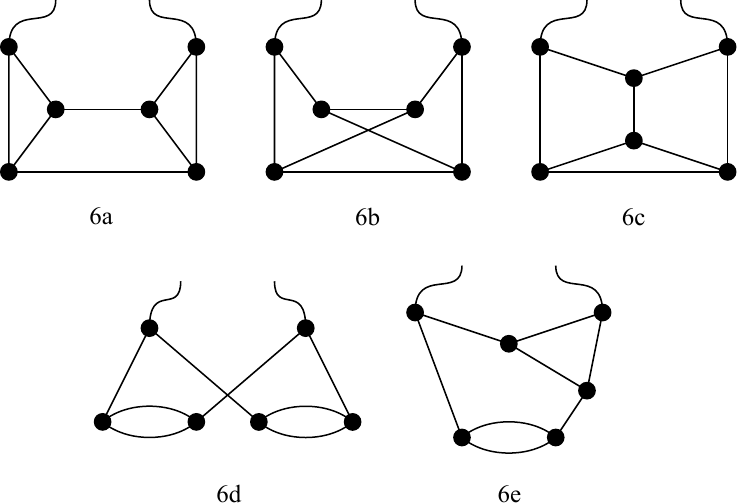}
\end{center}
\caption{Some cubic multigraphs on 6 vertices, with half-edges.}
\label{fig:pieces_6}
\end{figure}

To determine the disconnected cubic multigraphs on 8 vertices, we proceeded as follows. First, computer search found 24 connected cubic multigraphs on fewer than 8 vertices. Since there must be an even quantity of odd vertices, these are graphs on either 2, 4, or 6 vertices. In total, there are 2 such graphs on 2 vertices, 5 such graphs on 4 vertices, and 17 such graphs on 6 vertices; these are displayed in Figure \ref{fig:census_24} and Figure \ref{fig:census_6}. 

\begin{figure}
\begin{center}
\includegraphics[width=.9\textwidth]{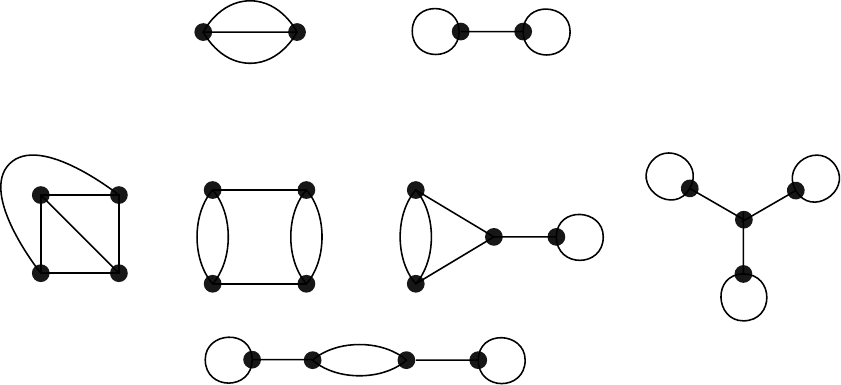}
\end{center}
\caption{The connected cubic multigraphs on  2 or 4 vertices.}
\label{fig:census_24}
\end{figure}

\begin{figure}
\begin{center}
\includegraphics[width=.9\textwidth]{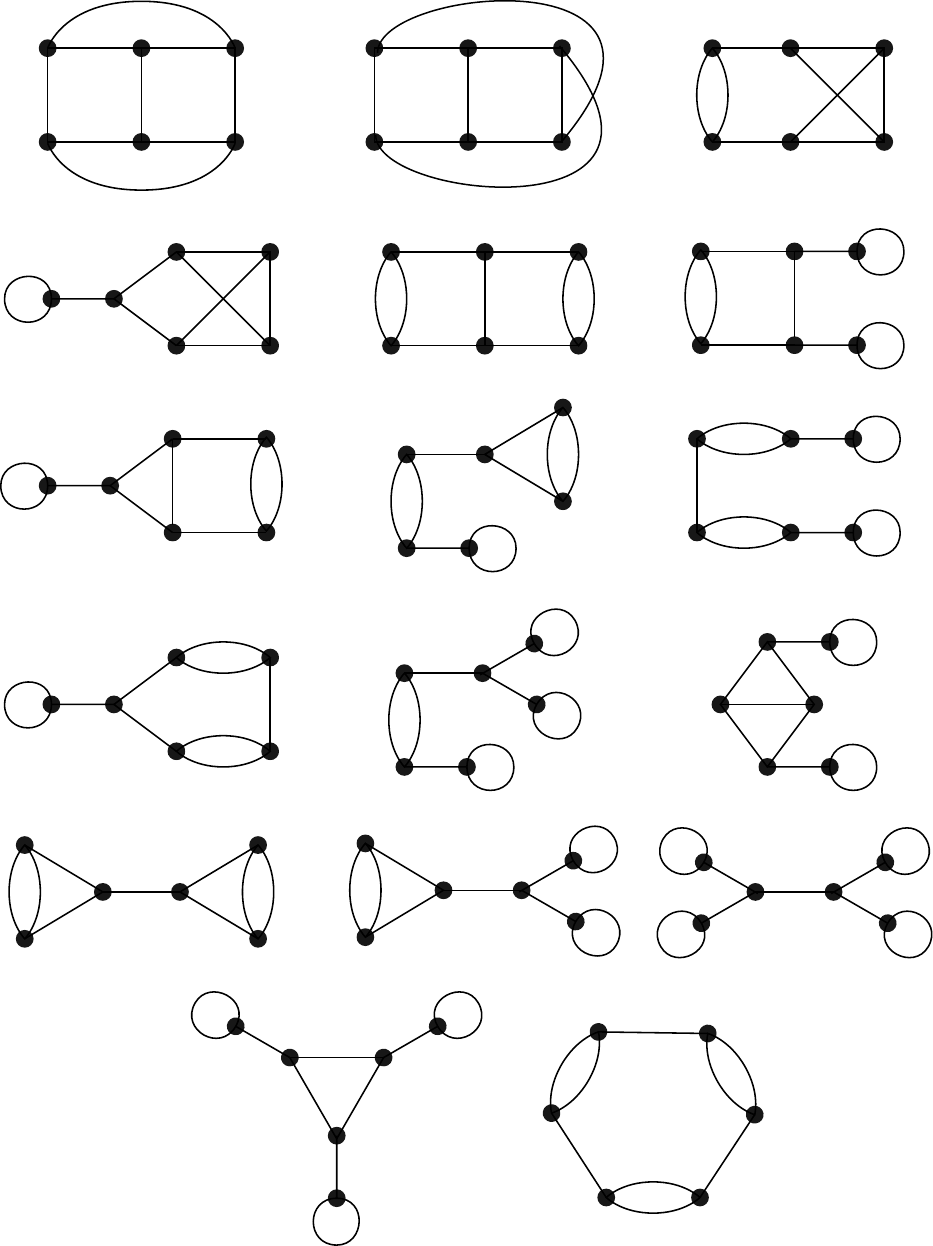}
\end{center}
\caption{The connected cubic multigraphs on 6 vertices.}
\label{fig:census_6}
\end{figure}

 Write $a$ and $b$ for the two connected cubic graphs on two vertices, write $F,F'$ for  arbitrary 4-vertex connected cubic graphs, and write $S$ for an arbitrary 6-vertex connected cubic graph. To enumerate the ways to build the disconnected cubic multigraphs on 8 vertices, we observe that the number 8 has 4 partitions, and consider the possible corresponding patterns.
\begin{itemize}
\item 8 = 2+2+2+2. This yields 5 possible patterns: 

$aaaa, aaab, aabb, abbb, bbbb$
\item 8 = 2+2+4; This yields 15 possible patterns: $aaF, abF, bbF$
\item 8 = 2+ 6; This yields 34 possible patterns: $aS, bS$
\item 8 = 4+4;  This yields 5 patterns: $FF$, and 10 patterns: $FF'$
\end{itemize}
We see that, overall, there are 69 disconnected cubic multigraphs on 8 vertices.

\subsection{Computation of quadrangulated immersions}\label{sec:computation}

The key takeaway from the computations we describe below is that we have found quandrangulated immersions of all but the three cubic graphs on 8 vertices shown in Figure \ref{fig:outlaws}. We conjecture that, in fact, none of these three graphs have any quadrangulated immersion at all.

\subsubsection{Exhaustive search}
After an earlier version of this article appeared on \texttt{arxiv.org}, Nico Van Cleemput significantly improved on the exhaustive search (by total number of vertices) we had carried out and described there.  Van Cleemput used plantri \cite{BrinkmannMckay07} combined with a plugin to generate quadrangulations having only cubic and quartic vertices, and then filtered those for the ones corresponding to immersions of cubic graphs \cite{VanCleemput}. For up to n=39 vertices, Van Cleemput ran the computations on an i5-4210U CPU running at 1.70GHz; finding all quadrangulated immersions with n=39 took 22.72 hours. For n=40,\ldots,46 the computation was run on a cluster of Intel Xeon Gold 6140 (Skylake) at 2.3 GHz. For n=46, 1082 distinct quadrangulations that are immersions of cubic graphs were found (corresponding to 90 distinct cubic graphs), which took about 1.15 CPU years.

\subsubsection{The two-disk construction}

Another computation was done using the two disks construction directly, which had the benefit of making it possible to consider graphs with hundreds of vertices. For example, presuming that all the irreducible immersions in Figure \ref{fig:disks} have been encoded, encoding an arbitrary immersion merely requires specifying an irreducible immersion and at most five additional natural numbers; see the example shown in Figure \ref{fig:encoded-disk-example}.
\begin{figure}
\begin{center}
\includegraphics[height=2in]{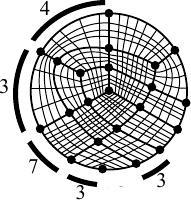}
\end{center}
\caption{A quadrangulated immersion of a disk encoded  by specifying the irreducible disk (in this case, the fourth disk from the left on the first row of Figure \ref{fig:disks}) and the five additional numbers 4, 3, 7, 3, 3, which specify the quantities and locations of the transversals. This example has 402 vertices. } 
\label{fig:encoded-disk-example}
\end{figure}

 To speed up the computation, the isomorphism check needed after each example was constructed was carried out as follows: Once a graph was determined to not be on the list of cubic graphs found so far, a hash was taken of each of the adjacency matrices for that graph over all permutations of its vertices, and saved for later use. Thus, checking whether a cubic graph had previously been found simply required comparing the hash of its one given adjaceny matrix to the list of known hashes.

This computation showed that quadrangulated immersions of 137 of the cubic multigraphs can be constructed using disks of circumference 18 or less. As before, the remaining three multigraphs are those shown in Figure \ref{fig:outlaws}. Further computation carried out on a cluster of Intel 8280 2.7 gHz processesors showed that these three examples cannot be achieved with the two disks construction using disks of circumference 48 or less. The computation for circumference 48  considered $5.44\times10^{10}$ cases and took about 1.97 CPU years.

\subsubsection{An additional computation}

An additional computation of interest, serving as a kind of test of some of the techniques described in this paper, used the {\tt NetworkX} library  to identify new quadrangulated immersions   from previously discovered ones by applying the radial method and reducing. Decomposing as per the two disks construction then provided new seeds for the two disks construction. This process was iterated 3 times, after which it no longer produced new quadrangulated immersions, yielding quadrangulated immersions of 133 cubic multigraphs. The 7 missing multigraphs were: the graph on the far right of Figure \ref{fig:census_3connected}, the graphs 5a\mys 3a, 5b\mys 3a, 5e\mys 3a (the last is shown in Figure \ref{fig:023}), and the graphs shown in Figure \ref{fig:outlaws}.

\begin{figure}
\begin{center}
\includegraphics[width=.97\textwidth]{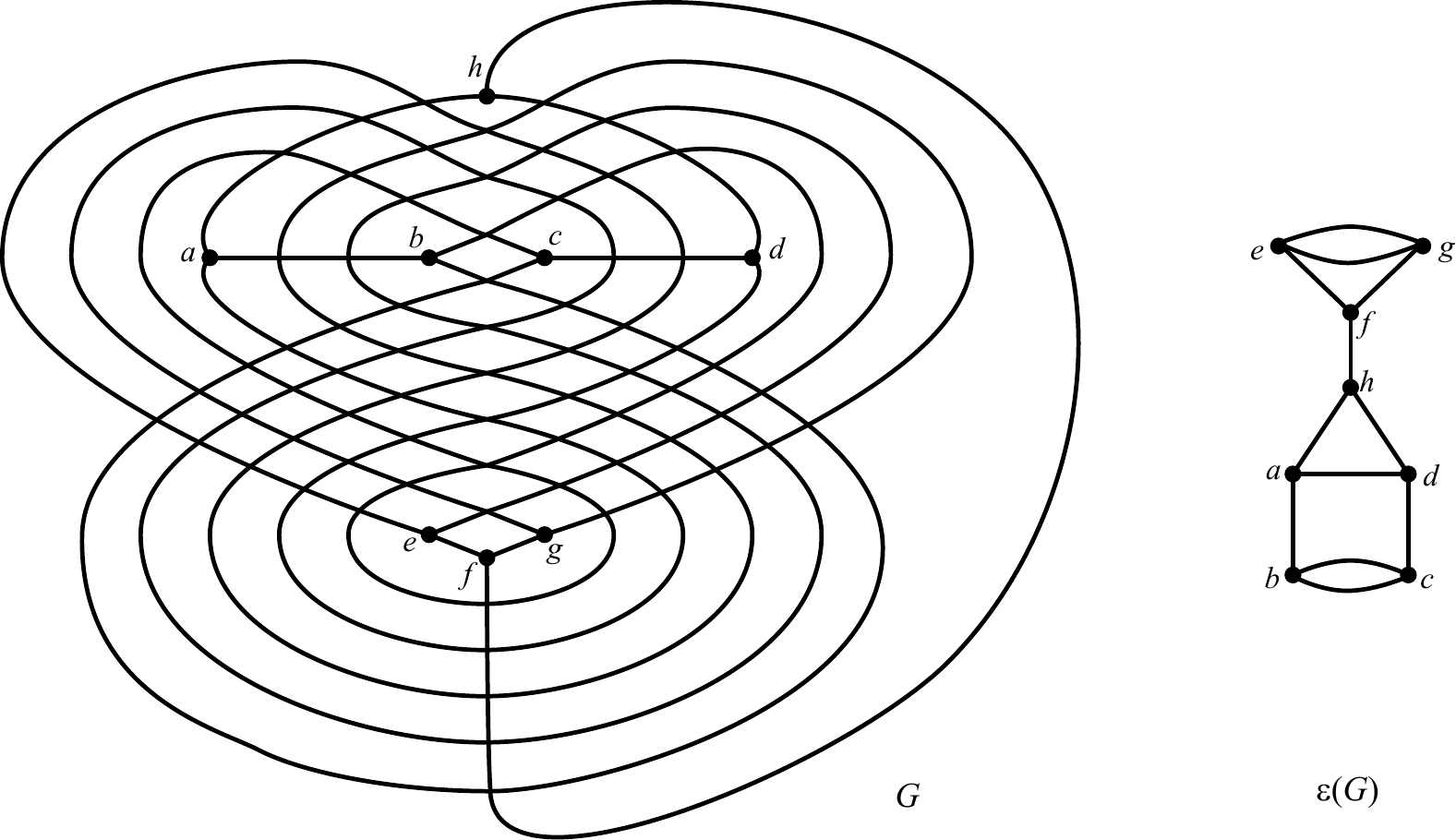}
\end{center}
\caption{A quadrangulated immersion, constructed by hand, which cannot be achieved with a two disks construction.}
\label{fig:023}
\end{figure}

\newpage
\providecommand{\bysame}{\leavevmode\hbox to3em{\hrulefill}\thinspace}
\providecommand{\MR}{\relax\ifhmode\unskip\space\fi MR }
\providecommand{\MRhref}[2]{%
  \href{http://www.ams.org/mathscinet-getitem?mr=#1}{#2}
}
\providecommand{\href}[2]{#2}

\bigskip


\end{document}